\documentclass[12pt,amssymb,verbatim]{amsart}

\usepackage{amsfonts,latexsym,geometry,color,hyperref,ulem}
\usepackage{amssymb,graphics,graphicx}
\usepackage{amsfonts,latexsym,geometry}
\usepackage{amssymb,graphics,graphicx}
\usepackage{bbm}
\usepackage{amscd}
\usepackage{amssymb}
\usepackage{amsthm}
\usepackage[utf8]{inputenc} 
\usepackage[T1]{fontenc}
\usepackage{amsmath}
\usepackage{amsfonts}
\usepackage{amssymb}
\usepackage{graphicx} 
\usepackage{enumerate}
\newcommand\supp{\mathrm{supp}}
\newtheorem{theoreme}{Theorem}[section] %
\newtheorem{proposition}[theoreme]{Proposition} %
\newtheorem{corollary}[theoreme]{Corollary} %
\newtheorem{lemme}[theoreme]{Lemma} %
\newtheorem{definition}{Definition}[section] %
\newtheorem{remark}[theoreme]{Remark} %
\newtheorem{conjecture}[theoreme]{Conjecture} %

\newcommand\sk{\smallskip}

\newcommand\R{\mathbb{R}}

\RequirePackage{yhmath} 
\renewcommand\widering[1]{\ring{#1}}

\author{E. Daviaud, Uliège}

\begin{document}
\title[Intrinsic Diophantine approximation]{Intrinsic Diophantine approximation by rationals of height with a bounded number of distinct prime factors}

\begin{abstract}
In the 80's Mahler suggested to study the Diophantine approximation by rationals lying in the middle-third Cantor set $K_{1/3}$. In this article,  writing, for every $N \in\mathbb{N},$ $\mathcal{P}_N$ the set of numbers with less than $N$ distinct prime factors, we show that, for any non-increasing mapping $\psi:\mathbb{N}\to \mathbb{R}_+,$ $$\dim_H \left\{x\in K_{1/3} : \ \vert x-\frac{p}{q}\vert\leq \psi(q)\text{ f.i.m. }\frac{p}{q}\in\mathbb{Q}\cap K_{1/3}, \ q\in\mathcal{P}_N \right\}=\frac{\dim_H K_{1/3}}{\max\left\{1,\delta_{\psi}\right\}},$$
where $\dim_H \cdot$ denotes the Hausdorff dimension and $\delta_{\psi}$ is the shrinking rate of $\psi.$ This result will be derived out a more general theorem regarding homogeneous self-similar rational IFS's. In addition, we establish several results regarding the intrinsic Diophantine approximation associated with inhomogeneous rational IFS's.
\end{abstract}
\maketitle

\section{Introduction}
The study of rational approximation on self-similar fractals was first suggested by Mahler~\cite{Mahler}. He asked how well numbers in the middle-third Cantor set can be approximated by rationals, and also by rationals lying in the middle-third Cantor set itself. Both of these questions have attracted much interest, and many results have recently been established.

Among these results, Bénard, He and Zhang \cite{KintFract} showed that if $\mu$ is a self-similar measure on $\mathbb{R}$ (Definition \ref{def-ssmu}), then for any non-increasing mapping $\psi: \mathbb{N}\to \mathbb{R}_+$ one has (where ``f.i.m.'' stands for ``for infinitely many'')
\begin{align*}
&\mu\Big(\Big\{x\in \mathbb{R}: \ |x-\tfrac{p}{q}|\leq \psi(q)\ \text{f.i.m. } \tfrac{p}{q}\in\mathbb{Q}\Big\}\Big)=1 
\ \Leftrightarrow \ \sum_{q\geq 1}q\psi(q)=+\infty,\\
&\mu\Big(\Big\{x\in \mathbb{R}: \ |x-\tfrac{p}{q}|\leq \psi(q)\ \text{f.i.m. } \tfrac{p}{q}\in\mathbb{Q}\Big\}\Big)=0 
\ \Leftrightarrow \ \sum_{q\geq 1}q\psi(q)<+\infty.
\end{align*}

Concerning the dimension of sets of points approximable at rate $\delta\geq 2$ (that is, when $\psi(q)=q^{-\delta}$, the rate $2$ being the threshold provided by Dirichlet's theorem) by rational numbers in the middle-third Cantor set $K_{1/3}$, Bugeaud and Durand~\cite{Bugdu} conjectured that
\begin{align*}
&\dim_H \Big\{x\in K_{1/3}: \ |x-\tfrac{p}{q}|\leq q^{-\delta}\ \text{f.i.m. } \tfrac{p}{q}\in\mathbb{Q}\Big\}\\
&=\max\Big\{\frac{\dim_H K_{1/3}}{\delta},\ \dim_H K_{1/3}+\frac{2}{\delta}-1\Big\}.
\end{align*}
Here, the first term corresponds to the intrinsic contribution (obtained by considering the approximation by rationals lying in the middle-third Cantor set), while the second corresponds to the extrinsic one. In this direction, Chow,Varjú and Yu~\cite{YuVarju} established that, in the case of the middle-seventh Cantor set, there exists $\delta_0>2$ such that the above formula holds for any $2\leq \delta\leq \delta_0$. This result was later generalized by Chen \cite{YuVarjuGene} and He and Liao \cite{HeLiaoJarnik} to self-similar sets satisfying the open set condition.

Regarding intrinsic approximation, namely approximation by rationals lying inside a given self-similar set, it is worth mentioning that it is not generic for a self-similar set to contain rational numbers. Indeed, given an homogeneous self-similar set $K\subset \mathbb{R}$, it is easily checked that for every $y\in\mathbb{R}$, $y+K$ is also the attractor of an homogeneous self-similar IFS. Moreover, one can show that
\[
\begin{cases}
\text{if }\mathcal{L}(K)>0\text{ then for a.e.\ }y\in\mathbb{R},\ \mathbb{Q}\cap (y+K)\neq \emptyset,\\
\text{if }\mathcal{L}(K)=0\text{ then for a.e.\ }y\in\mathbb{R},\ \mathbb{Q}\cap (y+K)= \emptyset.
\end{cases}
\]
The interested reader may refer to the appendix for a proof of this fact.

A natural class of self-similar sets known to contain rational numbers consists of those defined by rational parameters. The first results in this setting were obtained by Beresnevich and Velani~\cite{BV}, and refined by Levesley, Salp and Velani~\cite{LSV}, who provided a complete description of the metric theory of $\psi$-approximable points in base~3 in $K_{1/3}$. The problem was also studied, among others, by Fishman and Simmons~\cite{FishSim}, Tan, Wang and Wu~\cite{TWWTriad}, and Baker \cite{Baker}.
 Writing $h(r)$  the height of $r\in\mathbb{Q},$ that is, $h(r)=q$ if $r=\frac{p}{q}$ with $p \wedge q=1$, in~\cite{FishSim} and~\cite{TWWTriad}, the authors proposed the following conjecture (expressed differently),  extending that of Bugeaud and Durand\footnote{A careful reader will notice that the set investigated by Bugeaud and Durand does not involve the height $h(r)$. It as an exercise to show that, for any non increasing $\psi:\mathbb{N}\to \mathbb{R}_+,$ $\left\{x\in K_{1/3} : \ \vert x-\frac{p}{q}\vert\leq \psi(q)\text{ f.i.m. }(p,q)\in\mathbb{Z}\times \mathbb{N} \text{ s.t. }\frac{p}{q}\in K_T\right\}=\left\{x\in K_{1/3} : \ \vert x-r\vert\leq \psi(h(r))\text{ f.i.m. }r\in \mathbb{Q}\cap K_T\right\}\bigcup (\mathbb{Q}\cap K_T)$. Thus, the choice of the definition of the approximating sets does not change the dimension results.}: for any non-increasing mapping $\psi:\mathbb{N}\to \mathbb{R}_+$, writing
\begin{equation}
\label{DefDeltaPsi}
 \delta_{\psi}=\liminf_{q\to +\infty}\frac{-\log \psi(q)}{\log q},
\end{equation} 
\begin{align}
\label{ConjJArIntr}
\dim_H \Big\{x\in K_{1/3} : \ |x-r|\leq \psi(h(r))\ \text{f.i.m. } r\in \mathbb{Q}\cap K_{1/3}\Big\}=\frac{\dim_H K_{1/3}}{\max\Big\{1,\ \delta_{\psi}\Big\}}.
\end{align}

Given an integer $n\in\mathbb{N}$ who's decomposition in prime numbers is $n=\prod_{i=1}^{k_n}p_i^{\alpha_i},$ we write $$\omega(n)=k_n$$
and given $N \in\mathbb{N},$ we set $$\mathcal{P}_N =\left\{q\in\mathbb{N} :  \ \omega(q)\leq N\right\}.$$
In this article we make a step toward this conjecture by proving the following.

\begin{theoreme}
\label{ThmMahlerConj}
Let $b\geq 2$ and  $p_1,\dots,p_m \in\mathbb{Z}$ be integers and define
\[
T=\Big\{f_i(\cdot)=\frac{\cdot}{b}+\frac{p_i}{b}\Big\}_{1\leq i\leq m}.
\]
Let $K_T$ be the attractor of $T$ (see Definition \ref{def-ssmu}).  For any non-increasing $\psi:\mathbb{N}\to \mathbb{R}_+$, write
$$E_{\psi, T,N}=\Big\{x\in K_T : \ |x-r|\leq \psi(h(r))\ \text{f.i.m. } r\in \mathbb{Q}\cap K_T, \ h(r)\in\mathcal{P}_N  \Big\}.$$
If $K$ has empty interior, then
\begin{itemize}
\item[(1)] for every $N \geq \omega(b)+\omega(b-1),$ \[
\dim_H E_{\psi, T,N}=\frac{\dim_H K_T}{\max\{1,\delta_{\psi}\}},
\]
where $\delta_{\psi}$ is defined as in \eqref{DefDeltaPsi}.\medskip
\item[(2)] if there exists $p\in\mathbb{Z},k\in\mathbb{N}$ such that $\frac{p}{b^k}\in K_T,$ then  item $(1)$ holds for every $N\geq \omega(b).$
\end{itemize}

\end{theoreme}
Note the last item holds as soon as there exists $1\leq i\leq m$ such that $p_i =0\text{ or }b-1.$ In addition, when $b$ is a power of a prime, $\omega(b)=1.$ Hence, for the middle third Cantor set, item $(1)$ holds for every $N\geq 1.$
\begin{remark}
\label{Remequi}
For a self-similar IFS $T$ of the form $T=\Big\{f_i(\cdot)=\frac{\cdot}{b}+\frac{p_i}{b}\Big\}_{1\leq i\leq m}$, writing $K_T$ its attractor, the three following are equivalent:
\begin{itemize}
\item[(1)] $\mathcal{L}(K_T)>0,$
\medskip
\item[(2)] $K_T$ has non empty interior,
\medskip
\item[(3)] $\dim_H K_T =1.$
\end{itemize}
We refer to the last subsection of the appendix for a proof of this remark.
\end{remark}
In the above formula, $\frac{\dim_H K_T}{\max\left\{1, \delta_{\psi}\right\}}=\min\left\{\dim_H K_T ,\frac{\dim_H K_T}{\delta_{\psi}}\right\}.$
Moreover, it is a routine check to verify that $$\frac{\dim_H K_T}{\delta_{\psi}}=\inf\left\{s\geq 0 : \ \sum_{n\geq 1}e^{n \dim_H K_T}\psi\Big(e^{n}\Big)^s <+\infty\right\}.$$
Denote $\phi$ is the Euler mapping, that is $$\phi(\cdot)=\#\left\{t \leq n, \ t\wedge n =1\right\}$$
and notice that  for any $n\in\mathcal{P}_N,$
 $$\phi(n)\geq n \prod_{i=2}^{N+1}(1-\frac{1}{i})=C_{N}n.$$
In contrast with Theorem \ref{ThmMahlerConj},  when $K_T$ has non empty interior,  for any non-increasing $\psi:\mathbb{N}\to \mathbb{R}_+,$ a direct application of \cite{Maynkou} and \cite{BV}   yields 
$$\dim_H E_{\psi, T,N}=\min\left\{1,s_{\psi}\right\},$$
where $$s_{\psi}=\inf\left\{s\geq 0 : \ \sum_{n\in \mathcal{P}_N}n\psi(n)^s<+\infty\right\}.$$

To illustrate this striking dichotomy, we state the following corollary regarding the intrinsic Jarnik sets.
\begin{corollary}
Let $b\geq 2$  and $p_1,....,p_m \in\mathbb{Z}$ be integers and set $$T=\left\{f_i(\cdot)=\frac{\cdot}{b}+\frac{p_i}{b}\right\}_{1\leq i\leq m}.$$ Let $K_T$ be the attractor of $T$, $\delta \geq 1$ and  $\psi_{\delta}:q\mapsto q^{-\delta}.$ Then, for any $N \geq \omega(b)+\omega(b-1),$\footnote{ In the case where $K_T$ has non empty interior, the formula for $\dim_H E_{\psi_{\delta},T,N}$ is easily obtained from the expression of $s_{\psi}$ together with the fact that, for every large enough $x\in\mathbb{N},$ $$\#\left\{n\leq x : \ \omega(n)\leq N\right\}\geq \#\left\{n \leq x : n \text{ is prime}\right\}\geq \frac{x}{2\log x}.$$}
\begin{equation*}
\dim_H E_{\psi_\delta , T,N}=\begin{cases}\frac{\dim_H K_T}{\delta} \ \ \ \ \ \ \ \  \ \ \ \ \ \ \ \ \ \ \ \text{ if }K_T \text{ has empty interior } \\ \frac{2}{\max\left\{2,\delta \right\}}=\frac{2\dim_H K_T}{\max\left\{2,\delta\right\}} \ \ \ \ \  \text{ if }K_T \text{ has non empty interior.} \end{cases}
\end{equation*}

\end{corollary}
Regarding the original problem, we make a connection with the validity of the dimension formula  proposed in \eqref{ConjJArIntr} and a natural  conjecture in number theory. Given $q,b\in\mathbb{N}$ such that $q\wedge b=1,$ we write $O_q(b)$ the order of $b$ in $\mathbb{Z}_q ^*,$ that is, $$O_q(b)=\min\left\{k\geq 1 : \ b^k =1 \mod q \right\}.$$
\begin{conjecture}
\label{ConjectureO}
Let $b\geq 2$ be an integer and let $f:\mathbb{N}\to \mathbb{R}_+$ be a non-increasing mapping such that $f(q)\to 0.$ Then $$\lim_{x\to +\infty}\frac{\log\#\left\{1 \leq q \leq x : \ q\wedge b =1, \ O_q(b)\leq q^{f(q)}\right\}}{\log x}=0.$$
\end{conjecture}
\begin{theoreme}
\label{MahlConjO}
Let $b\geq 2$ and  $p_1,\dots,p_m \in\mathbb{Z}$ be integers and define
\[
T=\Big\{f_i(\cdot)=\frac{\cdot}{b}+\frac{p_i}{b}\Big\}_{1\leq i\leq m}.
\]
Let $K_T$ be the attractor of $T$.  For any non-increasing $\psi:\mathbb{N}\to \mathbb{R}_+$, write
$$E_{\psi, T}=\Big\{x\in K_T : \ |x-r|\leq \psi(h(r))\ \text{f.i.m. } r\in \mathbb{Q}\cap K_T\Big\}.$$
Assume that Conjecture \ref{ConjectureO} holds true and that  $K_T$ has empty interior, then \[
\dim_H E_{\psi, T}=\frac{\dim_H K_T}{\max\{1,\delta_{\psi}\}},
\]
where $\delta_{\psi}$ is defined as in \eqref{DefDeltaPsi}.
\end{theoreme}

If true, Conjecture \ref{ConjectureO} is probably difficult to prove. Nonetheless, we motivate this conjecture by proving a weaker version of it (which we use in the proof of Theorem \ref{ThmMahlerConj}).

\begin{theoreme}
\label{ConjOtilde}
Let $b\geq 2$ be an integer and let $f:\mathbb{N}\to \mathbb{R}_+$ be a non increasing mapping with $f(q)\to 0$. Then, for any $N\in\mathbb{N},$
 $$\lim_{x\to +\infty}\frac{\log\#\left\{1 \leq q \leq x : \ q\wedge b =1,\ \omega(q)\leq N \ O_{q}(b)\leq q^{f(q)}\right\}}{\log x}=0.$$
\end{theoreme}

The second purpose of this article is to provide more general results for inhomogeneous  rational IFSs. Given integers $m\geq 2$, $q_1,\dots,q_m\geq 2$ and $p_1,\dots,p_m\in\mathbb{Z}$, for $1\leq i\leq m$ define, for any $x\in\mathbb{R},$
\begin{equation}\label{Deffi}
f_i(x)=\frac{x}{q_i}+\frac{p_i}{q_i},
\end{equation}
let $T=\{f_1,\dots,f_m\}$ and denote by $K_T$ its attractor. We recall the following result, proved, which is due to Fishman and Simmons and Baker. 

\begin{proposition}[\cite{FishSim, Baker}]
\label{PropFSB}
Let $r\in\mathbb{Q}$. Then $r\in K_T$ if and only if there exist $\underline{i}=(i_1,\dots,i_n)\in \bigcup_{p\geq 1}\{1,\dots,m\}^p$ and $0\leq k\leq n-1$ such that
\[
r=\pi\Big((i_1,\dots,i_n,(i_{k+1},\dots,i_n)_\infty)\Big)=\pi\Big(i_1,\dots,i_n,i_{k+1},\dots,i_n,i_{k+1},...,i_n,...\Big),
\]
where $\pi$ is the canonical projection associated with $T$, define for any $(j_n)_{n\in\mathbb{N}}\in\left\{1,...,m\right\}^\mathbb{N}$ by $$\pi\Big((j_n)_{n\in\mathbb{N}}\Big)=\lim_{n\to +\infty}f_{j_1}\circ...\circ f_{j_n}(0). $$
\end{proposition}
As  Proposition \ref{PropFSB} was only proved in the case where $T$ satisfies the open set condition in \cite{FishSim} and mentioned to hold in \cite{Baker} , for the sake of completeness, Proposition \ref{PropFSB} is proved in the first subsection of the Appendix (see Subsubsection \ref{SubsecRatCod}).

A straightforward computation shows that there exists $p\in\mathbb{Z}$ such that
\[
\pi\Big((i_1,\dots,i_n,(i_{k+1},\dots,i_n)_\infty)\Big)
=
\frac{p}{\prod_{1\leq j\leq k}q_{i_j}\, \big(\prod_{k+1\leq j\leq n}q_{i_j}-1\big)}.
\]

One of the main difficulties in intrinsic Diophantine approximation is that there is, in general, no obvious way to reduce this fraction. Consequently, in~\cite{FishSim,Baker} the authors introduced a natural \emph{intrinsic height} for $r\in\mathbb{Q}\cap K_T$, which is in general larger than the usual height, by setting
\[
h_{\mathrm{int}}(r)
=
\min \Big\{\prod_{1\leq j\leq k}q_{i_j}\, \big(\prod_{k+1\leq j\leq n}q_{i_j}-1\big)\Big\},
\]
where the minimum is taken over all representations of $r$ of the above form. With this definition, the natural set to consider is
\[
E_{\psi,T,\text{int}}
=
\Big\{x\in K_T : \ |x-r|\leq \psi(h_{\mathrm{int}}(r))\ \text{f.i.m. } r\in\mathbb{Q}\cap K_T\Big\}.
\]

In~\cite{TWWTriad}, Tan, Wang and Wu provided a complete metric description of $E_{\psi,T,\text{int}}$ for the middle-third Cantor set, and in~\cite{Baker}, Baker gave conditions under which these sets have full measure with respect to the natural self-similar measure on $K_T$. In the present article, we provide a complete description of the Hausdorff dimension of the sets $E_{\psi,T,\text{int}}$.

\begin{theoreme}
\label{ThmIntrin}
Let $T$ be the IFS defined by~\eqref{Deffi}, $K_T$ its attractor, and let $\psi:\mathbb{N}\to \mathbb{R}_+$ be non-increasing. Then
\[
\dim_H E_{\psi,T,\text{int}}
=
\frac{\dim_H K_T}{\max\Big\{1,\delta_{\psi}\Big\}}.
\]
\end{theoreme}

It is worth mentioning that it is not known in general whether $\mathcal{H}^{\dim_H K_T}(K_T)>0$ for IFSs of the form~\eqref{Deffi}. Moreover, this result follows from a more general theorem concerning approximation by ultimately periodic sequences on self-similar fractals (see Theorem \ref{ThmUPAWSC}).

Finally, as for any $r\in\mathbb{Q}\cap K_T,$ $h_{\text{int}}(r)\geq h(r)$, the following corollary holds.

\begin{corollary}
\label{CorollaryLowB}
Let $T$ be a self-similar IFS defined by maps $f_i$ as in~\eqref{Deffi}. Then for any non-increasing $\psi:\mathbb{N}\to \mathbb{R}_+$, 
\begin{align*}
&\dim_H \Big\{x\in K_T : \ |x-\tfrac{p}{q}|\leq\psi(q)\ \text{f.i.m. } \tfrac{p}{q}\in \mathbb{Q}\cap K_T \Big\}\geq \frac{\dim_H K_T}{\max\left\{1,\delta_{\psi}\right\}}.
\end{align*}
\end{corollary}
\medskip

In the first section, we recall some background on geometric measure theory,  self-similar IFS's and the separation properties they satisfy when the parameters defining these IFS's are rationals. Since Theorem~\ref{ThmIntrin} provides the appropriate lower bound for Theorem~\ref{ThmMahlerConj} (see Corollary~\ref{CorollaryLowB}), the proof of Theorem~\ref{ThmMahlerConj} is split into two parts. Section~\ref{SecMahler} establishes the “upper-bound” part of Theorem~\ref{ThmMahlerConj}, while Section~\ref{SecUP} proves Theorem~\ref{ThmIntrin}. Finally, Section~\ref{Secpersp} presents some open problems and perspectives related to the topic.

\section{Notations and recalls}
\label{SecRecall}

Let us start with some notation 

Given $k\in\mathbb{N},$ we write $\mathbb{N}_{\geq k}=\left\{k,k+1,...\right\}$.
\smallskip

 Let $d$ $\in\mathbb{N}$. For $x\in\mathbb{R}^{d}$, $r>0$,  $B(x,r)$ stands for the closed ball of ($\mathbb{R}^{d}$,$\vert\vert \ \ \vert\vert_{\infty}$) of center $x$ and radius $r$. 
 Given a ball $B$, $\vert B\vert$ stands for the diameter of $B$. For $t\geq 0$, $\delta\in\mathbb{R}$ and $B=B(x,r)$,   $t B$ stands for $B(x,t r)$, i.e. the ball with same center as $B$ and radius multiplied by $t$,   and the  $\delta$-contracted  ball $B^{\delta}$ is  defined by $B^{\delta}=B(x ,r^{\delta})$.
\smallskip

Given a set $E\subset \mathbb{R}^d$, $\widering{E}$ stands for the  interior of the set $E$, $\overline{E}$ its  closure and $\partial E$ its boundary, i.e, $\partial E =\overline{E}\setminus \widering{E}$. If $E$ is a Borel subset of $\R^d$, its Borel $\sigma$-algebra is denoted by $\mathcal B(E)$.
\smallskip

Given a topological space $X$, the Borel $\sigma$-algebra of $X$ is denoted $\mathcal{B}(X)$ and the space of probability measures on $\mathcal{B}(X)$ is denoted by $\mathcal{M}(X).$ 

\sk
 The $d$-dimensional Lebesgue measure on $(\mathbb R^d,\mathcal{B}(\mathbb{R}^d))$ is denoted by 
$\mathcal{L}^d$.
\smallskip

For $\mu \in\mathcal{M}(\R^d)$,   $\supp(\mu)=\left\{x\in \mathbb{R}^d: \ \forall r>0, \ \mu(B(x,r))>0\right\}$ is the topological support of $\mu$.
\smallskip

Given two integers $n,m$ one will write $n\vert m$ if there exists $k\in\mathbb{Z}$ such that $m=nk$.
\smallskip

Given $f:\mathbb{N}\to \mathbb{R}_{+,}$ we will write $f(q)=o_{q\to +\infty}(1)$ if $\lim_{q\to +\infty} \frac{f(q)}{q}=0.$
\smallskip

\subsection{Recalls on geometric measure theory and fractal geometry}

We now recall some definitions.

\begin{definition}
\label{hausgau}
Let $\zeta:\mathbb{R}^{+}\to\mathbb{R}^+$ be increasing in a neighborhood of $0$. The Hausdorff outer measure at scale $t\in(0,+\infty]$ associated with $\zeta$ of a set $E$ is defined by
\begin{equation}
\label{gaug}
\mathcal{H}^{\zeta}_t(E)=\inf \left\{\sum_{n\in\mathbb{N}}\zeta(|B_n|):\ |B_n|\leq t,\ B_n \text{ closed balls and } E\subset \bigcup_{n\in \mathbb{N}}B_n\right\}.
\end{equation}
The Hausdorff measure associated with $\zeta$ of a set $E$ is defined by
\[
\mathcal{H}^{\zeta}(E)=\lim_{t\to 0^+}\mathcal{H}^{\zeta}_t(E).
\]
\end{definition}

For $t\in (0,+\infty]$, $s\geq 0$ and $\zeta(x)=x^s$, we write $\mathcal{H}^{\zeta}_t(E)=\mathcal{H}^{s}_t(E)$ and $\mathcal{H}^{\zeta}(E)=\mathcal{H}^{s}(E)$, which are respectively called the $s$-dimensional Hausdorff outer measure at scale $t$ and the $s$-dimensional Hausdorff measure.

\begin{definition}
\label{dim}
Let $\mu\in\mathcal{M}(\mathbb{R}^d)$. For $x\in \supp(\mu)$, the lower dimension of $\mu$ at $x$ is defined by
\[
\underline\dim_{\rm loc}(\mu,x)=\liminf_{r\to 0^+}\frac{\log\mu(B(x,r))}{\log r}.
\]
The  Hausdorff dimension of $\mu$ is defined by
\begin{equation}
\label{dimmu}
\dim_H(\mu)=\mathrm{ess\,inf}_{\mu}(\underline\dim_{\rm loc}(\mu,x)).
\end{equation}
\end{definition}

It is known (see~\cite{F} for details) that
\begin{align*}
&\dim_H(\mu)=\inf\{\dim_H(E):\,E\in\mathcal{B}(\mathbb{R}^d),\ \mu(E)>0\}.
\end{align*}

\subsection{Some geometric properties of self-similar iterated function systems}
\subsubsection{Basic definitions and facts}

A map $f:\mathbb{R}^d\to\mathbb{R}^d$ is called a contracting similarity if there exist $\lambda\in(-1,1)\setminus \left\{0\right\}$, $O\in\mathcal{O}_d(\mathbb{R})$ and $a\in\mathbb{R}^d$ such that
\[
f(x)=\lambda Ox+a.
\]
The number $c=|\lambda|$ is called the contraction ratio of $f$.

\begin{definition}
\label{def-ssmu}
A self-similar IFS is a family $S=\{f_i\}_{1\leq i\leq m}$ of $m\geq 2$ contracting similarities of $\mathbb{R}^d$. Let $(p_i)_{1\leq i\leq m}$ be a probability vector.

The associated self-similar measure $\mu$ is the unique probability measure satisfying
\begin{equation}
\label{def-ssmu2}
\mu=\sum_{i=1}^m p_i\,\mu\circ f_i^{-1}.
\end{equation}
Its support is the attractor $K$, i.e.\ the unique non-empty compact set such that
\[
K=\bigcup_{i=1}^m f_i(K).
\]
\end{definition}

The existence and uniqueness of $K$ and $\mu$ are standard~\cite{Hutchinson}, and by a result of Feng--Hu~\cite{FH}, self-similar measures are exact dimensional. Moreover, it was established by Falconer in \cite{Dimh=DimB} that any self-similar set $K$ satisfies 
\begin{equation}
\label{DimHEqDimB}
\dim_H K =\dim_B K,
\end{equation}
where $\dim_B$ denotes the box dimension.

Given a self-similar IFS $T=\{f_1,\dots,f_m\}$, we write
\[
\begin{cases}
\Lambda=\{1,\dots,m\},\\
\Lambda^{*}=\bigcup_{k\geq 1}\Lambda^k,\\
0<c_1,\dots,c_m<1 \text{ the contraction ratios of } f_1,\dots,f_m,\\
\text{for }\underline{i}=(i_1,\dots,i_n)\in\Lambda^*,\ c_{\underline{i}}=c_{i_1}\cdots c_{i_n},\\
\text{and } f_{\underline{i}}=f_{i_1}\circ\cdots\circ f_{i_n}.
\end{cases}
\]

\subsubsection{Separation properties of rational IFSs}
\label{SecSS}
Let $T=\{f_1,\dots,f_m\}$ be a self-similar IFS, let $K$ be its attractor, and let
$0<c_1\leq \cdots\leq c_m<1$ be the contraction ratios of $f_1,\dots,f_m$.
For $r>0$, we define
\[
\Lambda_r=\Big\{\underline{i}=(i_1,\dots,i_n)\in \bigcup_{k\geq 1}\{1,\dots,m\}^k:\ 
c_{\underline{i}}\leq r< c_{(i_1,\dots,i_{n-1})}\Big\}.
\]

We begin by recalling the definition of the Asymptotically Weak Separation Condition (AWSC). This property was introduced by Feng in \cite{Feng2007} and applies to a very large class of IFSs (see \cite[Theorem 1.3]{BFmult}), as mentioned in the introduction.

\begin{definition}
\label{DefAWSC}
We say that $T$ satisfies the AWSC if
\[
\lim_{r\to 0^+}\frac{\log \Big(\sup_{x\in K} 
\#\big\{f_{\underline{i}}:\ \underline{i}\in \Lambda_r \text{ and } f_{\underline{i}}(K)\cap B(x,r)\neq \emptyset\big\} \Big)}{\log r}=0.
\]
\end{definition}

The following result is proved in~\cite[Theorem~1.5]{BFmult}.

\begin{theoreme}[\cite{BFmult}]
Let $q_1,\dots,q_m\in\mathbb{N}$ and $p_1,\dots,p_m\in\mathbb{Z}$. Then the IFS
\[
T=\Big\{f_i(\cdot)=\frac{\cdot}{q_i}+\frac{p_i}{q_i}\Big\}_{1\leq i\leq m}
\]
satisfies the AWSC.
\end{theoreme}

When the above IFS $T$ is homogeneous, i.e.\ when $q_1=\cdots=q_m$, it satisfies a stronger separation condition.

\begin{definition}
\label{DefWSC}
Let $T=\{f_1,\dots,f_m\}$ be a self-similar IFS. We say that $T$ satisfies the Weak Separation Condition (WSC) if
\[
\sup_{x\in K,\ r>0}
\#\big\{f_{\underline{i}}:\ \underline{i}\in \Lambda_r \text{ and } f_{\underline{i}}(K)\cap B(x,r)\neq \emptyset\big\}<+\infty.
\]
\end{definition}
For references regarding IFS's satisfying the weak separation conditions, see for instance \cite{DefWSC,ZernerWSC}.
Given $q\in\mathbb{N}$ and $p_1,\dots,p_m\in\mathbb{Z}$, it is straightforward to check that
\[
T=\Big\{f_i(\cdot)=\frac{\cdot}{q}+\frac{p_i}{q}\Big\}_{1\leq i\leq m}
\]
satisfies the WSC. Moreover, Zerner established the following result.

\begin{theoreme}[\cite{ZernerWSC}]
\label{ThmZerner}
Let $T$ be a self-similar IFS satisfying the WSC and let $K_T$ be its attractor. Then
\[
0<\mathcal{H}^{\dim_H K_T}(K_T)<+\infty.
\]
\end{theoreme}

For $0\leq s\leq d,$ recall that a measure $\mu\in \mathcal{M}(\mathbb{R}^d)$ is called $s$-Alfhors-regular if there exists a constant $C>0$ such that for every $x\in\supp(\mu)$ and every $r>0,$ one has $$ C^{-1}r^s \leq \mu\Big(B(x,r)\Big)\leq C r^s.$$
Although it is not stated explicitly, it is a direct consequence of Theorem~\ref{ThmZerner} that if $T$ satisfies the WSC, then $K_T$ is Ahlfors regular (i.e. there exists a $\dim_H K_T$-Alfors-regular measure $\mu \in\mathcal{M}(\mathbb{R}^d)$ with $\supp(\mu)=K_T$). For the sake of completeness, we prove this fact in the appendix.

\begin{proposition}
\label{WSCAlfh}
Let $T$ be a self-similar IFS satisfying the WSC. Then $K_T$ is $\dim_H K_T$-Ahlfors regular.
\end{proposition}

\section{Proof of Theorem \ref{ThmMahlerConj} and Theorem \ref{MahlConjO}}

\label{SecMahler}
Let $m,b\geq 2$, let $p_1,\dots,p_m\in\mathbb{Z}$, and let
\[
T=\Big\{f_i(\cdot)=\frac{\cdot}{b}+\frac{p_i}{b}\Big\}.
\]
Denote by $K_T$ the attractor of $T$. 

Let $\widetilde{K}_T =K_T \text{ mod }1\subset \mathbb{T}^1$ be the projection of $K_T$ onto the one-dimensional torus. It is easily verified that if $x\in \widetilde{K}_T$, $bx\in \widetilde{K}_T$. Moreover, $\widetilde{K}_T$ is compact; hence, if $\widetilde{K}_T\neq \mathbb{T}^1$, then $\mathbb{T}^1\setminus \widetilde{K}_T$ is a non-empty open set, while if $\widetilde{K}_T=\mathbb{T}^1$, then $K_T$ has non-empty interior.

We will show the following.

\begin{theoreme}
\label{MainThMahl}
Let $S\subset \mathbb{T}^1$ be a $\times b$-invariant set.
Assume that $\mathbb{T}^1\setminus S$ has non-empty interior. Then,   there exists $Q \in\mathbb{N}$ such that for every $q\geq Q,$ $q\wedge b=1,$ $$(\exists p \in\mathbb{Z}, p\wedge q=1 : \ \frac{p}{q}\in S)\Rightarrow O_q(b)\leq q^{f(q)},$$
with $f(q)=\frac{C}{\sqrt{\log \log \log q}},$ where  $C>0$ is a constant.
\end{theoreme}

In the next subsection, we explain how Theorem~\ref{ThmMahlerConj} and Theorem \ref{MahlConjO} are  derived from Theorem \ref{MainThMahl}, Theorem \ref{ConjOtilde} and Conjecture \ref{ConjectureO}. In Subsection \ref{SecConjOtilde}, we prove Theorem \ref{ConjOtilde}. The two next subsections are dedicated to the proof of Theorem \ref{MainThMahl}. Subsection \ref{SecRecEqui} recall basic facts and results about equi-distribution theory and Theorem \ref{MainThMahl} is established in Subsection \ref{SecProMal}. 
\subsection{Proof of Theorem \ref{ThmMahlerConj} and Theorem \ref{MahlConjO} assuming Theorems \ref{MainThMahl}, \ref{ConjOtilde} and \ref{ThmIntrin}}
\subsubsection{Proof of the upper-bound}

Given $S \subset \mathbb{T}^1$, a $\times b$ invariant set and $ \mathcal{N}$ a subset of $\mathbb{N},$ define 
\[
Q_{b,n,\mathcal{N}}(S)=\Big\{q\in \mathcal{N}, b^{n-1}<q\leq b^n:\ \exists p\in\mathbb{N},\ \frac{p}{q}\in S,\ p\wedge q=1\Big\}
\]
and $Q_{b,n}=Q_{b,n,\mathbb{N}}.$

Theorem \ref{MainThMahl} together with Theorem \ref{ConjOtilde} and Conjecture \ref{ConjectureO} yields the following corollary.

\begin{corollary}
\label{MainCoroMahl}
Let $S\subset \mathbb{T}^1$ be a $\times b$-invariant set.
Assume that $\mathbb{T}^1\setminus S$ has non-empty interior. Then:
\begin{itemize}
\item[(1)]  for every $N\in \mathbb{N},$ \[
\lim_{n\to +\infty}\frac{\log \#Q_{b,n, \mathcal{P}_N}(S)}{n}=0.
\]
\item[(2)] assume that Conjecture \ref{ConjectureO} holds true, then \[
\lim_{n\to +\infty}\frac{\log \#Q_{b,n}(S)}{n}=0.
\]
\end{itemize}

\end{corollary}

We start by establishing the following upper-bound
\begin{proposition}
Let $S \subset \mathbb{T}^1$ be a $\times$b invariant set and $\mathcal{N}\subset \mathbb{N}$ such that for any  $q \in\mathbb{N}$  ($q'$ divides $q$) $\Rightarrow$ $(q' \in \mathcal{N}).$ Assume that $$\lim_{n\to +\infty}\frac{\log \# Q_{b,n,\mathcal{N}}}{n}=0.$$
Then, for any non-increasing mapping $\psi:\mathbb{N}\to \mathbb{R}_+$, $$\dim_H \left\{x\in \mathbb{T}^1 : \ \vert x-\frac{p}{q}\vert\leq \psi(q)\text{ f.i.m. }q \in \mathcal{N},p \in\mathbb{Z}: \ \frac{p}{q}\in S\right\} \leq \frac{\overline{\dim}_B S}{\max\left\{1,\delta_{\psi}\right\}},$$
where $\overline{\dim}_B(S)$ denotes the upper box dimension of the set $S$. 
\end{proposition}
The upper-bound for Theorem \ref{ThmMahlerConj} and Theorem \ref{MahlConjO} would then readily follow from Corollary  \ref{MainCoroMahl} with $\mathcal{N}=\mathcal{P}_N$ and $\mathcal{N}=\mathbb{N}.$

\begin{proof}
Without loss of generality, we may assume that $\delta_{\psi}\geq 1.$ Observe that, by monotonicity of $\psi,$ 
\begin{align*}
&\left\{x\in \mathbb{T}^1 : \ \vert x-\frac{p}{q}\vert\leq \psi(q)\text{ f.i.m. }q \in \mathcal{N},p \in\mathbb{Z}: \ \frac{p}{q}\in S\right\}\\
=&\left\{x\in \mathbb{T}^1 : \ \vert x-\frac{p}{q}\vert\leq \psi(q)\text{ f.i.m. }q \in \mathcal{N},p \in\mathbb{Z}, p\wedge q=1: \ \frac{p}{q}\in S\right\} .
\end{align*}
Fix $\varepsilon>0$ and $r>0$ small enough so that, for every set $\mathcal{C}$  of pairwise disjoint balls centered on $S$ of radii $0<\rho<r,$ $$\#\mathcal{C}\leq \Big(\frac{1}{\rho}\Big)^{\overline{\dim}_B S +\varepsilon}.$$
As for any, $q\in\mathbb{N},$ for any  $p_1 \neq p_2$ such that $\frac{p_1}{q},\frac{p_2}{q}\in S,$ one has $$B(\frac{p_1}{q},\frac{1}{2q})\cap B(\frac{p_2}{q},\frac{1}{2q})=\emptyset,$$
for every $q>\frac{2}{r},$ $$\#\left\{p\in\mathbb{Z}: \ \frac{p}{q}\in S\right\}\leq (2q)^{\overline{\dim}_B S+\varepsilon}  $$

Remark also that for any $N\in\mathbb{N}$,
\begin{align*}
&\left\{x\in \mathbb{T}^1 : \ \vert x-\frac{p}{q}\vert\leq \psi(q)\text{ f.i.m. }q \in \mathcal{N},p \in\mathbb{Z}, p\wedge q=1: \ \frac{p}{q}\in S\right\}\\
& \subset \bigcup_{n\geq N}\ \bigcup_{\substack{q\in Q_{b,n,\mathcal{N}}(S)\\ p:\ \frac{p}{q}\in S}}
B\Big(\frac{p}{q},\psi(b^n)\Big).
\end{align*}

Assume that Theorem~\ref{MainThMahl} holds. Fix $\varepsilon>0$ and choose $N_\varepsilon$ large enough so that for every $k\geq N_\varepsilon$ and every
$\widetilde{k}\geq \frac{\log N_\varepsilon}{\log b}$,
\[
\begin{cases}
\psi(k)\leq k^{-(1-\varepsilon)\delta_\psi}<r,\\
\#Q_{b,\widetilde{k},\mathcal{N}}(S)\leq b^{\varepsilon \widetilde{k}}.
\end{cases}
\]

Set
\[
s_\varepsilon=\frac{\overline{\dim}_B S+2\varepsilon}{\delta_\psi}+\varepsilon.
\]
Then,
\begin{align*}
\sum_{n\geq \frac{\log N_\varepsilon}{\log b}}\sum_{q\in Q_{b,n,\mathcal{N}}(S)}\sum_{p:\ \frac{p}{q}\in S}\psi(q)^{s_\varepsilon}
&\leq 2^{\overline{\dim}_B(S)+\varepsilon}\sum_{n\geq \frac{\log N_\varepsilon}{\log b}}
 b^{n(\overline{\dim}_B S +\varepsilon)}\, b^{\varepsilon n}\, b^{-n(1-\varepsilon)\delta_\psi s_\varepsilon}\\
&=2^{\overline{\dim}_B(S)+\varepsilon}\sum_{n\geq \frac{\log N_\varepsilon}{\log b}}
b^{-n(\delta_\psi s_\varepsilon-\overline{\dim}_B S-2\varepsilon)}\\
&\leq 2^{\overline{\dim}_B(S)+\varepsilon} \sum_{n\geq \frac{\log N_\varepsilon}{\log b}} b^{-n\varepsilon\delta_\psi}<+\infty.
\end{align*}
Hence 
$$\left\{x\in \mathbb{T}^1 : \ \vert x-\frac{p}{q}\vert\leq \psi(q)\text{ f.i.m. }q \in \mathcal{N},p \in\mathbb{Z}, p\wedge q=1: \ \frac{p}{q}\in S\right\}\leq   s_\varepsilon$$
and letting $\varepsilon\to 0$ yields the claim.

\end{proof}

\subsubsection{Proof of the lower-bound}
Notice that the lower-bound for Theorem \ref{MahlConjO} is implied by Corollary \ref{CorollaryLowB}. Thus we only prove the lower-bound required to establish Theorem \ref{ThmMahlerConj}. We focus first on item $(2)$ of Theorem \ref{ThmMahlerConj}. Let us fix $N\geq \omega(b).$ 
Let $r=\frac{p}{b^k}\in K_T$ be a $b$-adic number. Remark that,  for every $n\in\mathbb{N}$ and every $\underline{i}=(i_1,...,i_n)\in \left\{1,...,m\right\}^n,$ there exists $p' \in\mathbb{Z}$ such that $$f_{\underline{i}}(r)=\frac{p'}{b^{n+k}},$$
hence $$\omega(h(f_{\underline{i}}(r)))\leq \omega(b).$$

As $N \geq \omega(b),$
\begin{align*}
&\left\{x\in K_T : \ \vert x-\frac{p}{q}\vert\leq \psi(q)\text{ f.i.m. }q \in \mathcal{P}_{N},p \in\mathbb{Z}, p\wedge q=1: \ \frac{p}{q}\in K_T\right\}\\
&\supset \left\{x\in \mathbb{T}^1 : \ \vert x-\frac{p}{b^t}\vert\leq \psi(q)\text{ f.i.m. }t \in \mathbb{N},p \in\mathbb{Z},  \ \frac{p}{b^t}\in K_T\right\}\\
&=\limsup_{t \in \mathbb{N},p \in\mathbb{Z},  \ \frac{p}{b^t}\in K_T}B\Big(\frac{p}{b^t},\psi(b^t)\Big).
\end{align*}

Moreover, the argument provided in \cite[P.1821]{ED3} yields
$$\dim_H \limsup_{n \geq 1, \underline{i}=(i_1,...,i_n)\in \left\{1,...,m\right\}^n}B(f_{\underline{i}}(r),\psi(n))\geq \frac{\dim_H K_T}{\delta_{\psi}}$$
so that
$$\dim_H \left\{x\in K_T : \ \vert x-\frac{p}{q}\vert\leq \psi(q)\text{ f.i.m. }q \in \mathcal{P}_{N},p \in\mathbb{Z}, p\wedge q=1: \ \frac{p}{q}\in K_T\right\}\geq   \frac{\dim_H K_T}{\delta_{\psi}}.$$

We now establish item $(1):$

Without loss of generality, assume that $p_1 \neq 0$ (otherwise we can apply item $(2)$) and write $$r_0 =\pi\Big((1)_{\infty}\Big)=\sum_{k\geq 1}p_1 b^{-k}=\frac{p_1}{b-1}\in K_T \cap \mathbb{Q}.$$
For every  $n\in\mathbb{N}$ and every $\underline{i}=(i_1,...,i_n)\in \left\{1,...,m\right\}^n,$ there exists $p' \in\mathbb{Z}$ such that $$f_{\underline{i}}(r_0)=\frac{p'}{b^{n}(b-1)}.$$
Since $\omega(h( f_{\underline{i}}(r_0)))\leq \omega(b^n(b-1))= \omega(b)+\omega(b-1),$ the same argument as above yields the desired conclusion.

\medskip

In the next section, we establish Theorem \ref{ConjOtilde}.

\subsection{Proof of Theorem \ref{ConjOtilde}}
\label{SecConjOtilde}

 For $b\geq 2,$ $q\in\mathbb{N}$ with $q\wedge b=1,$ writing $q=\prod_{1\leq i\leq \omega(q)}p_i^{\alpha_i}$ the  decomposition in prime factors of $q$, we set 
$$\widetilde{O}_{q}(b)=\prod_{1\leq i\leq \omega(q)}O_{p_i^{\alpha_i}}(b).$$

One will show the following.

\begin{theoreme}
\label{ConjOtildebis}
Let $\varepsilon>0$ be a real number. For every $x \in\mathbb{N}$ large enough, $$\#\left\{0 \leq q \leq x : \ q\wedge b=1, \ \widetilde{O}_q(b)\leq q^{\varepsilon}\right\}\leq x^{2\varepsilon}.$$
\end{theoreme}

Prior to proving Theorem \ref{ConjOtildebis}, we explain how it implies Theorem \ref{ConjOtilde}.

Remark that, for any $q\in\mathbb{N}$, $b\wedge q=1$ one has 
$$ O_{q}(b)\leq \widetilde{O}_q(b)\leq O_q(b)^{\omega(q)}.$$
Let $N\in\mathbb{N}$ be an integer. As $f(q)\to 0,$ for every large enough $q$ such that $\omega(q)\leq N$,  $$O_q(b)\leq q^{f(q)}\Rightarrow \widetilde{O}_q(b)\leq q^{N f(q)}\leq q^{\varepsilon}.$$
 Thus, Theorem \ref{ConjOtildebis} yields $$\limsup_{x\to +\infty}\frac{\#\left\{0 \leq q \leq x : \ q\wedge b=1, \ \omega(q)\leq N \ O_q(b)\leq q^{\varepsilon}\right\}}{\log x}\leq 2\varepsilon$$
and letting $\varepsilon\to 0$ proves Theorem \ref{ConjOtilde}.

We now prove Theorem \ref{ConjOtildebis} and we start by a lemma.

\begin{lemme}
Let $p$ be a prime number, coprime with $b$. Write $k_p=O_p(b)$ and $u_p=v_p(b^{k_p}-1),$ the valuation of $p$ in $b^{k_p}-1.$ Then, for every $t\in\mathbb{N},$ 
$$\begin{cases}O_{p^t}(b)=k_p\text{ if }t\leq u_p \\ O_{p^t}(b)=k_p p^{t-u_p}\text{ if }t\geq u_p.\end{cases} $$
\end{lemme}
\begin{proof}

For $t\leq u_p,$   $p^t \vert b^{k_p}-1$ and $O_{p^t}(b)\geq k_p$ yields $$O_{p^t}(b)=k_p.$$
Let us now fix $t>u_p.$ Notice that 
\begin{align*} 
b^{k_p}=1+m p^{u_p},
\end{align*}
 with $m \wedge p=1$ (otherwise $p^{u_p+1}\vert b^{k_p}-1$). Hence, there exists $\gamma \in\mathbb{N} $ such that
 \begin{align}
 \label{Equarec}
 b^{pk_p}=1+m^p p^{u_p+1} +\gamma p^{u_p+2}.
\end{align}  
This yields $O_{p^{u_p+1}}(b) \vert pk_p.$ Since  $O_{p^{u_p+1}}(b)>k,$  $O_{p^{u_p+1}}(b)=pk_p.$ Moreover, in \eqref{Equarec},  $m^p \wedge p=1,$ so that the same argument applies to $O_{p^{u_p+2}}(b)$ and, recursively, to any $O_{p^{k}}(b)$, $k\geq u_p$.
\end{proof}

Call $\mathcal{P}^b$ the set of prime numbers coprime with $b$. For $t\in\mathbb{R},$ we set, when defined, 

$$S_t =\prod_{p\in\mathcal{P}^b}\Big(\sum_{k\geq 0}\frac{1}{O_{p^k}(b)^t}\Big),$$
which develops in
 $$S_t =\sum_{n\geq 1, \ n\wedge b=1}\frac{1}{\widetilde{O}_n (b)^t}.$$
For every $p$ prime, write again $k_p =O_p(b)$ and $u_p=v_p(b^{k_p}-1).$ Then $$\sum_{i\geq 0}\frac{1}{O_{p^i}(b)^t}=\sum_{k=0}^{u_p} \frac{1}{k_p ^t}+\sum_{i\geq u+1}\frac{1}{(k_p p^{i-u})^t}\leq \frac{u_p}{k_p ^t}\Big(\sum_{k\geq 0}\frac{1}{p^{kt}}\Big).$$
As $p^{u_p}\vert b^{k_p} -1$, for every $p>b$, $u_p \leq k_p$. Hence, for any such $p$ $$\frac{u_p}{k_p ^t}\Big(\sum_{k\geq 0}\frac{1}{p^{kt}}\Big)\leq  \sum_{k\geq 0}\frac{1}{p^{kt}}$$
so that $S_t$ converges for every $t>1$.

Fix $1<t<2$ and assume that, for $x\in\mathbb{N},$
$$\#\left\{1 \leq n\leq x : \ n\wedge b=1, \  \widetilde{O}_n(b)\leq n^{\varepsilon}\right\}> x^{2\varepsilon}.$$
Then $$\sum_{1 \leq n \leq x : \ n \wedge b =1} \frac{1}{\widetilde{O}_n(b)^t} \geq \frac{x^{2 \varepsilon}}{x^{t\varepsilon}}=x^{2\varepsilon -t\varepsilon}.$$
As $x^{(2-t)\varepsilon}\to +\infty,$ this happens for an at most finite set of $x$.

\medskip

We now proceed to prove Theorem~\ref{MainThMahl}. The proof is based on the following observation. Fix $\frac{p}{q}\in S$ with $p\wedge q=1$ and $q\wedge b=1$. On the one hand, since $S$ is $\times b$ invariant, for any $k\in\mathbb{N}$ we have $b^k\frac{p}{q}\in S$; on the other hand, the sequence $b^k p$ should be equidistributed modulo $q$, unless the multiplicative order of $b$ modulo $q$ is very small (see Theorem \ref{ThmBourg} below). Since the complement of $S$ contains an interval, only those $q$ for which $b$ has small multiplicative order $\text{mod } q$ can satisfy $\frac{p}{q}\in S$.
\smallskip

In the next subsection, we recall basic definitions and tools related to the discrepancy of sequences on $\mathbb{T}^1$. We then establish Theorem~\ref{MainThMahl}.

\subsection{Recall on equidistribution}
\label{SecRecEqui}
We begin by recalling the definition of discrepancy.

\begin{definition}
Let $U\subset \mathbb{T}^1$ be a finite set with $\#U=N<+\infty$. The discrepancy of $U$ is defined by
\[
D(U)=\sup_{0<a<b<1}\Big| \frac{\#\{u\in U:\ u\in [a,b]\}}{N}-(b-a)\Big|.
\]
\end{definition}

\begin{remark}
From the definition, it is immediate that if $D(U)\leq \eta$ and $J\subset \mathbb{T}^1$ is an interval with $\mathcal{L}(J)>\eta$, then there exists $u\in U$ such that $u\in J$.
\end{remark}

We also recall the Erd\H{o}s--Tur\'an inequality.

\begin{theoreme}[\cite{ErdTur}]
\label{ErdosTuran}
Let $U$ be a finite sequence with $\#U=N$. Then, for any $A\in\mathbb{N}$,
\[
D(U)\leq C\Big(\frac{1}{A}+\sum_{1\leq j\leq A}\frac{1}{Nj}\Big|\sum_{u\in U}\exp(2\pi i j u)\Big|\Big).
\]
\end{theoreme}

\subsection{Proof of Theorem~\ref{MainThMahl}}
\label{SecProMal}
We fix $b$ and $S$ as in the theorem, and denote by $\ell>0$ the length of an interval
$I\subset \mathbb{T}^1\setminus S$. By considering a larger constant $C$ in Theorem \ref{ErdosTuran} if necessary, we may assume that $\frac{C}{2\ell}\in\mathbb{N}$.

Given $q\in\mathbb{N}$, we write $\mathbb{Z}_q=\mathbb{Z}/q\mathbb{Z}$ and $\mathbb{Z}_q^*$ for its group of invertible elements.
For $a\in\mathbb{Z}_q^*$, we denote by $o_q(a)$ the multiplicative order of $a$ modulo $q$.

We first recall the following result due to Bourgain.

\begin{theoreme}[\cite{BourgainSumExp}]
\label{ThmBourg}
Let $q\in\mathbb{N}$ be an integer and let $H$ be a subgroup of $\mathbb{Z}_q^*$. There exists $C>0$ such that, provided that $q$ is large enough,
\begin{equation}
\label{equabourgain}
\#H\geq q^{\frac{C}{\sqrt{\log \log \log q}}}
\quad\Longrightarrow\quad
\max_{p\in\mathbb{Z}_q^*}
\Big|\sum_{h\in H}\exp\Big(2\pi i \frac{hp}{q}\Big)\Big|
\leq o_{q\to+\infty}(1)\#H .
\end{equation}
\end{theoreme}

We prove the following proposition.

\begin{proposition}
\label{LEmmaqepsi}
For any large enough $n\in\mathbb{N}$ and any
$q\in Q_{b,n}(S)$ with $q\wedge b=1$, there exists
$1\leq t\leq \frac{2C}{\ell}$  such that
\[
O_{q/(q\wedge t)}(b)<\Big(\frac{q}{q\wedge t}\Big)^{f(\frac{q}{q\wedge t})},
\]
where $f(p)=\frac{C}{\sqrt{\log \log \log p}},$ with $C$ as in Theorem \ref{ThmBourg}. 
\end{proposition}
Notice that for every $m\in\mathbb{N}$, for every $q$ large enough, $f(\frac{q}{m})\leq 2f(q).$ Moreover, as, for any $q\in\mathbb{N},$ $q\wedge b=1,$ for any $p$ dividing $q$ $$O_{q}(b)\leq O_{q/p}(b)\times O_p(b)\leq pO_{q/p}(b),$$

 Proposition \ref{LEmmaqepsi} implies, with the notation involved, that for $q$ large enough, $$O_{q}(b)\leq \frac{\frac{2C}{\ell}}{(q\wedge t)^{f(\frac{q}{q\wedge t})}}q^{f(\frac{q}{q\wedge t})}\leq q^{3 f(q)},$$
yielding Theorem \ref{MainThMahl}.

\begin{proof}
Assume on the contrary that for every $1\leq t\leq \frac{2C}{\ell}$,
\[
o_{q/(q\wedge t)}(b)\geq \Big(\frac{q}{q\wedge t}\Big)^{f(\frac{q}{q\wedge t})} .
\]
Since $q\in Q_{b,n}(S)$, there exists $p\in\mathbb{Z}_q^*$ such that $\frac{p}{q}\in S$.
As $S$ is $\times b$-invariant, this implies that
\[
U=\Big\{u_k:=\frac{b^k p}{q}\Big\}_{0\leq k\leq o_q(b)-1}\subset S .
\]

Applying  Erd\H{o}s--Tur\'an 's inequality (Theorem \ref{ErdosTuran}) with $A=\frac{2C}{\ell}$, we obtain
\[
D(U)\leq \frac{\ell}{2}
+\sum_{1\leq h\leq \frac{2C}{\ell}}
\frac{1}{h O_q(b)}
\Big|\sum_{0\leq k\leq O_q(b)-1}
\exp\Big(2\pi i \frac{hp b^k}{q}\Big)\Big|.
\]

For $1\leq h\leq \frac{2C}{\ell}$, set $q'=\frac{q}{h\wedge q}$ and $h'=\frac{h}{h\wedge q}$. Then
\[
\Big|\sum_{0\leq k\leq O_q(b)-1}\exp\Big(2\pi i \frac{hp b^k}{q}\Big)\Big|
=
\Big|\frac{O_q(b)}{O_{q'}(b)}
\sum_{0\leq k\leq O_{q'}(b)-1}
\exp\Big(2\pi i \frac{b^k p h'}{q'}\Big)\Big|.
\]
Hence
\[
\frac{1}{O_q(b)}\Big|\sum_{0\leq k\leq O_q(b)-1}\exp\Big(2\pi i \frac{hp b^k}{q}\Big)\Big|
=
\frac{1}{O_{q'}(b)}
\Big|\sum_{0\leq k\leq O_{q'}(b)-1}
\exp\Big(2\pi i \frac{b^k ph'}{q'}\Big)\Big|.
\]

Since $O_{q'}(b)\geq {q'}^{f(q')}$, by Theorem \ref{ThmBourg}, provided that $q$ is large enough, 
\[
\frac{1}{O_{q'}(b)}
\Big|\sum_{0\leq k\leq O_{q'}(b)-1}
\exp\Big(2\pi i \frac{b^k ph'}{q'}\Big)\Big|
\leq o_{q\to+\infty}(1).
\]
Therefore,
\[
D(U)\leq \frac{\ell}{2}+\Big(\frac{2C}{\ell}\Big)^{2}\,o_{q\to+\infty}(1).
\]
For $q$ large enough this implies $D(U)<\ell$, hence there exists
$0\leq k\leq O_q(b)-1$ such that $u_k\in I\subset \mathbb{T}^1\setminus S$, yielding a contradiction.
\end{proof}

\section{Proof of Theorem~\ref{ThmIntrin}}
\label{SecUP}
\subsection{Approximation by eventually periodic sequences}

Fix $T=\{f_1,\dots,f_m\}$ a self-similar IFS, let $K$ be its attractor, and let
$0<c_1\leq \cdots \leq c_m<1$ be the contraction ratios of $f_1,\dots,f_m$.
Theorem~\ref{ThmIntrin} will be derived from the following general result.

\begin{theoreme}
\label{ThmUPAWSC}
Let $\eta:\mathbb{R}_+\to \mathbb{R}_+$ be non-decreasing and assume that $T$ satisfies AWSC. Writing
\[
\delta_{\eta}=\liminf_{r\to 0^+}\frac{\log \eta(r)}{\log r},
\]
one has
\begin{align*}
\dim_H \Big\{y\in K :\ &
\big\| y-\pi\Big((i_1,\dots,i_n,(i_{k+1},\dots,i_n)_\infty)\Big)\big\|_\infty\\
&\leq \eta\Big(\Big\vert f_{(i_1,\dots,i_n)}(K)\Big\vert\Big)\ \text{i.o.}\Big\}
=
\frac{\dim_H K}{\max\{1,\delta_\eta\}},
\end{align*}
where i.o.\ means that the inequality holds for infinitely many words
$(i_1,\dots,i_n)\in \bigcup_{p\geq 1}\{1,\dots,m\}^p$ and $1\leq k\leq n$.
\end{theoreme}

We now explain how this applies to intrinsic Diophantine approximation.

Let $\psi:\mathbb{R}_+\to \mathbb{R}_+$ be non-increasing and define
\[
W_{T,\psi}
=
\Big\{x\in K : |x-r|\leq \psi(h_{\mathrm{int}}(r))\ \text{f.i.m. } r\in\mathbb{Q}\cap K\Big\}.
\]

\begin{lemme}
Setting $\eta(r)=\psi(1/r)$, the following inclusions hold:
\begin{align*}
E_{\psi, T, \text{int}}
&\subset
\widetilde{E}_{\eta, T}:=\limsup_{\underline{i}=(i_1,\dots,i_n)\in\Lambda^*}
\bigcup_{0\leq k\leq n-1}
B\Big(\pi\big(i_1,\dots,i_n,(i_{k+1},\dots,i_n)_\infty\big),\\
&\hspace{3cm}
\eta\Big(\frac{1}{\prod_{1\leq j\leq k}q_{i_j}\,(\prod_{k+1\leq j\leq n}q_{i_j}-1)}\Big)\Big)\\
&\subset E_{\psi,T,\text{int}}\cup (\mathbb{Q}\cap K).
\end{align*}
\end{lemme}

\begin{proof}
The first inclusion follows from the fact that for any $r\in\mathbb{Q}\cap K$ there exists
$\underline{i}=(i_1,\dots,i_n)\in\{1,\dots,m\}^n$ such that
\[
h_{\mathrm{int}}(r)
=
\prod_{1\leq j\leq k}q_{i_j}\,(\prod_{k+1\leq j\leq n}q_{i_j}-1).
\]

Moreover, if
\[
r=\pi\big(i_1,\dots,i_n,(i_{k+1},\dots,i_n)_\infty\big),
\]
then, by monotonicity of $\eta$,
\[
\eta\Big(\frac{1}{h_{\mathrm{int}}(r)}\Big)
\geq
\eta\Big(\frac{1}{\prod_{1\leq j\leq k}q_{i_j}\,(\prod_{k+1\leq j\leq n}q_{i_j}-1)}\Big).
\]
Hence, if $x$ belongs to infinitely many of the above balls, then $x\in E_{\psi,T,\text{int}} $.

Conversely, if
$x$ belongs to the limsup set but only finitely many distinct rationals
$r=\pi(i_1,\dots,i_n,(i_{k+1},\dots,i_n)_\infty)$ satisfy the corresponding inequality, then there exists
$r_x\in\mathbb{Q}\cap K$ and infinitely many $\underline{i}$ such that
\[
x\in B\Big(r_x,\eta\Big(\frac{1}{\prod_{1\leq j\leq k}q_{i_j}\,(\prod_{k+1\leq j\leq n}q_{i_j}-1)}\Big)\Big),
\]
which implies $x=r_x\in\mathbb{Q}$.
\end{proof}

\begin{corollary}
\begin{align*}
\dim_H E_{\psi,T,\text{int}}=\dim_H
\widetilde{E}_{\eta, T}.
\end{align*}
\end{corollary}

Finally, for any $\underline{i}=(i_1,\dots,i_n)$,
\[
|f_{\underline{i}}(K)|
=\frac{|K|}{\prod_{1\leq j\leq n}q_{i_j}}
\leq
\frac{|K|}{\prod_{1\leq j\leq k}q_{i_j}(\prod_{k+1\leq j\leq n}q_{i_j}-1)}
\leq
\frac{2|K|}{\prod_{1\leq j\leq n}q_{i_j}}.
\]

Define $\widetilde{\eta},\widehat{\eta}:\mathbb{R}_+\to\mathbb{R}_+$ by
\[
\widetilde{\eta}(r)=\eta\Big(\frac{r}{|K|}\Big),
\qquad
\widehat{\eta}(r)=\eta\Big(\frac{2r}{|K|}\Big).
\]
Then $\delta_{\widetilde{\eta}}=\delta_{\widehat{\eta}}=\delta_\eta$ and
\[
\widetilde{\eta}(\vert f_{\underline{i}}(K)\vert)
\leq
\eta\Big(\frac{1}{\prod_{1\leq j\leq k}q_{i_j}(\prod_{k+1\leq j\leq n}q_{i_j}-1)}\Big)
\leq
\widehat{\eta}(\vert f_{\underline{i}}(K)\vert).
\]
Applying Theorem~\ref{ThmUPAWSC} to $\widetilde{\eta}$ and $\widehat{\eta}$ yields the desired result.
\subsection{Proof of Theorem \ref{ThmUPAWSC}}

First, note that the case $\delta_{\eta}<1$ readily follows from the case $\delta_{\eta}\geq 1$. Thus we assume throughout that $\delta_{\eta}\geq 1$.

\subsubsection{Proof of the upper bound}
\label{Secproofupper}
We introduce the following notation related to the possible occurrence of exact overlaps:
\begin{equation}
\begin{cases}
cl(\underline{i})=\{\underline{j}\in\Lambda^n : f_{\underline{j}}=f_{\underline{i}}\},\\
\mathcal{CL}_n=\{cl(\underline{i}) : \underline{i}\in\Lambda^n\}.
\end{cases}
\end{equation}

Given $\underline{i}=(i_1,\ldots,i_n)\in\Lambda^n$, define
\[
\text{\emph{Uper}}(\underline{i})
=
\Big\{
\pi\big((j_{k+1},\ldots,j_n)_\infty\big):
\underline{j}=(j_1,\ldots,j_n)\in cl(\underline{i}),\ 1\le k\le n
\Big\}.
\]

We start with a counting lemma.

\begin{lemme}\label{lemmecountUP}
For every $\varepsilon>0$ there exists $N_{\varepsilon}\in\mathbb{N}$ such that for every $n\geq N_{\varepsilon}$ and any $\underline{i}\in\{1,\ldots,m\}^n$,
\[
1\le \#\mathrm{Uper}(\underline{i})\le |f_{\underline{i}}(K)|^{-\varepsilon}.
\]
\end{lemme}

\begin{proof}
Set $\varepsilon_0=\varepsilon/2$. First observe that for any $\underline{j}\in cl(\underline{i})$ and any $1\le k\le n$,
\begin{align*}
\pi\big((j_1,...,j_n),(j_{k+1},\ldots,j_n)_\infty\big)
&=\lim_{t\to +\infty}f_{(j_1,...,j_n)}\circ f_{(i_{k+1},...,i_n,...,i_{k+1},...,i_n)}(0)\\
&=
f_{\underline{i}}\Big(\pi\big((j_{k+1},\ldots,j_n)_\infty\big)\Big).
\end{align*}
Hence
\[
\#\text{\emph{Uper}}(\underline{i})
=
\#\Big\{\pi\big((j_{k+1},\ldots,j_n)_\infty\big):
\underline{j}\in cl(\underline{i}),\ 1\le k\le n
\Big\}.
\]

Moreover, if $(\widetilde{j}_1,...,\widetilde{j}_n),(j_1,..,j_n)\in \mathcal{C}l(\underline{i})$ are such that $$f_{(\widetilde{j}_1,....,\widetilde{j}_k)}=f_{(j_1,...,j_k)},$$
one has  $f_{(\widetilde{j}_{k+1},....,\widetilde{j}_n)}=f_{(j_{k+1},...,j_n )},$ so that $$\pi\big((\widetilde{ j}_{k+1},\ldots,\widetilde{j}_n)_\infty\big)=\pi\big((j_{k+1},\ldots,j_n)_\infty\big).$$

This yields $$\#\text{\emph{Uper}}(\underline{i})
\leq
\#\Big\{f_{(j_1,...,j_k)}:
\underline{j}=(j_1,...,j_n)\in cl(\underline{i}),\ 1\le k\le n
\Big\}. $$

Fix $0<t<c_1$ and set
\[
n_{\underline{i}}=\Big\lfloor\frac{-\log c_{\underline{i}}}{\log t}\Big\rfloor+1,
\qquad
0\le p\le n_{\underline{i}}.
\]
Define
\[
\widetilde{\text{\emph{Per}}}(\underline{i},p)
=
\Big\{
f_{(j_1,...,j_k)}:
(j_1,\ldots,j_n)\in cl(\underline{i}),\ 1\le k\le n,\
t^{-(p+1)}<c_{(j_{1},\ldots,j_k)}\le t^{-p}
\Big\}.
\]
Fix $x\in f_{\underline{i}}(K)$ and remark that for any $\underline{j}\in \mathcal{C}l(\underline{i}),$
\[
f_{(j_1,\ldots,j_k)}(K)\cap f_{\underline{i}}(K)\neq\varnothing.
\]
As $$\widetilde{\text{\emph{Per}}}(\underline{i},p)\subset \big\{f_{\underline{i}}:\ \underline{i}\in \Lambda_{\vert K \vert t^p} \text{ and } f_{\underline{i}}(K)\cap B(x,\vert K \vert t^p)\neq \emptyset\big\},$$ 
by Definition \ref{DefAWSC}, there exists a constant $C_{\varepsilon_0}>0$ such that $$\#\widetilde{\text{\emph{Per}}}(\underline{i},p) \leq C_{\varepsilon_0} t^{-\varepsilon_0 p}\leq\frac{C_{\varepsilon_0}}{\vert K\vert}\vert f_{\underline{i}}(K)\vert^{-\varepsilon_0}.$$
Since $$ \bigcup_{0\leq p \leq n_{\underline{i}}}\widetilde{\text{\emph{Per}}}(\underline{i},p)=\widetilde{\text{\emph{Per}}}(\underline{i}),$$
one has $$\#\widetilde{\text{\emph{Per}}}(\underline{i},p)\leq (n_{\underline{i}}+1)\frac{C_{\varepsilon_0}}{\vert K\vert}\vert f_{\underline{i}}(K)\vert^{-\varepsilon_0} \leq  \vert f_{\underline{i}}(K)\vert^{-2\varepsilon_0}$$
provided that $n$ was chosen large enough.
\end{proof}
Let $B$ be a ball and denote by $r=\operatorname{rad}(B)$ its radius. Set
\[
\Lambda_B=
\Big\{
\underline{i}=(i_1,\ldots,i_n)\in\Lambda^*:f_{\underline{i}}(K)\cap B\neq \emptyset, 
c_{\underline{i}}\le r< c_{(i_1,\ldots,i_{n-1})}
\Big\}.
\]

Recall that $\dim_H K=\dim_B K$, where $\dim_B$ denotes the box dimension of $K$.
Fix $0<\lambda<c_1$, where $c_1=\min_{1\le i\le m}c_i$.
Let $\varepsilon>0$ and define
\[
s_\varepsilon=
\frac{\dim_H K+2\varepsilon}{(1-\varepsilon)\delta_\eta}+\varepsilon.
\]
Choose $N_\varepsilon\in\mathbb N$ such that the following hold:

\begin{itemize}
\item[(1)]
Lemma~\ref{lemmecountUP} holds for all $\underline{i}\in\{1,\ldots,m\}^n$ with $n\ge N_\varepsilon$;

\item[(2)]
for every $n\ge N_\varepsilon$ (recalling \eqref{DimHEqDimB}) there exists a covering $\mathcal C_n$ of $K$ by balls of radius $\lambda^n$ satisfying
\[
\#\mathcal C_n\le \lambda^{-n\dim_H K-\varepsilon n};
\]

\item[(3)]
for every ball $B$  with $|B|\le \lambda^{N_\varepsilon}$ (recalling that $T$ satisfies AWSC (Definition \ref{DefAWSC}),
\[
\#\{f_{\underline{i}}:\underline{i}\in\Lambda_B\}\le |B|^{-\varepsilon};
\]

\item[(4)]
for any $0<r\le\lambda^{N_\varepsilon}$,
\[
\eta(r)\le r^{\delta_\eta(1-\varepsilon)}.
\]
\end{itemize}

For $n\ge N_\varepsilon$ and $L\in\mathcal C_n$, by (1) and (3) we have
\[
\#\{f_{\underline{i}}:\underline{i}\in\Lambda_L\}\le |L|^{-\varepsilon},
\qquad
\#\text{\emph{Uper}}(\underline{i})\le |L|^{-\varepsilon}
\quad\text{for all }\underline{i}\in\Lambda_L,
\]
where we used that $\underline{i}\in\Lambda_L$ implies
$|f_{\underline{i}}(K)|\asymp |L|$.

Moreover, since $0<\lambda<c_1$, for every
$\underline{i}\in\bigcup_{n\ge N_\varepsilon}\{1,\ldots,m\}^n$
there exists $L\in\bigcup_{n\ge N_\varepsilon}\mathcal C_n$ such that
$\underline{i}\in\Lambda_L$.

Therefore, for any $k\ge N_\varepsilon$,
\begin{align*}
&\limsup_{\underline{i}=(i_1,\ldots,i_n)\in\Lambda^*}
\bigcup_{0\le k\le n-1}
B\Big(
\pi((i_{k+1},\ldots,i_n)_\infty),
\eta(|f_{\underline{i}}(K)|)
\Big)
\\
&\subset
\bigcup_{n\ge k}
\bigcup_{L\in\mathcal C_n}
\bigcup_{\underline{i}\in\Lambda_L}
\bigcup_{x\in\text{\emph{Uper}}(\underline{i})}
B\big(x,\eta(|f_{\underline{i}}(K)|)\big).
\end{align*}

Furthermore,
\begin{align*}
\sum_{n\ge N_\varepsilon}
\sum_{L\in\mathcal C_n}
\sum_{\underline{i}\in\Lambda_L}
\#\text{\emph{Uper}}(\underline{i})\,\eta(|L|)^{s_\varepsilon}
&\le
\sum_{n\ge N_\varepsilon}
\#\mathcal C_n\,\lambda^{-2\varepsilon n}
\,\lambda^{n s_\varepsilon\delta_\eta(1-\varepsilon)}
\\
&\le
\sum_{n\ge N_\varepsilon}
\lambda^{-n\dim_H K-2\varepsilon n+s_\varepsilon n\delta_\eta(1-\varepsilon)}
\\
&\le
\sum_{n\ge N_\varepsilon}\lambda^{n\varepsilon}
<\infty.
\end{align*}

Consequently,
\[
\dim_H
\limsup_{\underline{i}=(i_1,\ldots,i_n)\in\Lambda^*}
\bigcup_{0\le k\le n-1}
B\Big(
\pi((i_{k+1},\ldots,i_n)_\infty),
\eta(|f_{\underline{i}}(K)|)
\Big)
\le s_\varepsilon.
\]

Since this holds for every $\varepsilon>0$, we finally obtain
\[
\dim_H
\limsup_{\underline{i}=(i_1,\ldots,i_n)\in\Lambda^*}
\bigcup_{0\le k\le n-1}
B\Big(
\pi((i_{k+1},\ldots,i_n)_\infty),
\eta(|f_{\underline{i}}(K)|)
\Big)
\le
\frac{\dim_H K}{\delta_\eta}.
\]

\subsubsection{Proof of the lower bound}

Fix $\varepsilon>0$ and let $(r_k)_{k\in \mathbb{N}}$ be a decreasing sequence in
$\mathbb{R}_+^{\mathbb{N}}$ with $r_k\to 0$ such that, for every $k\in\mathbb{N}$,
\[
\eta(r_k)\geq r_k^{(1+\varepsilon)\delta_{\eta}}.
\]

Recall that for any $k\in\mathbb{N}$,
\[
K=\bigcup_{\underline{i}\in\Lambda_{r_k}}f_{\underline{i}}(K).
\]
Moreover, for any $\underline{i}=(i_1,\dots,i_n)\in\Lambda_{r_k}$,
\[
f_{\underline{i}}(K)\subset
B\Big(\pi\big((i_1,\dots,i_n),(i_n)_\infty\big),|f_{\underline{i}}(K)|\Big).
\]

This implies that
\[
K=\limsup_{k\in\mathbb{N},\,\underline{i}=(i_1,\dots,i_n)\in\Lambda_{r_k/c_1}}
B\Big(\pi\big((i_1,\dots,i_n),(i_n)_\infty\big),\frac{r_k}{c_1}\Big).
\]

We recall the following result, established as \cite[Proposition 6.7]{ED4}.

\begin{proposition}[\cite{ED4}]
There exist a self-similar IFS $T_\varepsilon$ and an associated self-similar
measure $\mu_\varepsilon$ such that
\[
\begin{cases}
\supp(\mu_\varepsilon)=K,\\
\dim_H \mu_\varepsilon\geq \dim_H K-\varepsilon.
\end{cases}
\]
\end{proposition}

The following mass transference principle is established in \cite{ED3}.

\begin{theoreme}[\cite{ED3}]
Let $\mu\in\mathcal{M}(\mathbb{R})$ be a self-similar measure and let
$(B_n)_{n\in\mathbb{N}}$ be a sequence of balls centred on $\supp(\mu)$ such that
\[
\mu\Big(\limsup_{n\to+\infty}B_n\Big)=1.
\]
Then, for any $\delta\geq 1$,
\[
\dim_H \limsup_{n\to +\infty}B_n^\delta
\geq \frac{\dim_H \mu}{\delta}.
\]
\end{theoreme}

Since
\[
\mu_\varepsilon\Big(
\limsup_{k\in\mathbb{N},\,\underline{i}=(i_1,\dots,i_n)\in\Lambda_{r_k/c_1}}
B\Big(\pi\big((i_1,\dots,i_n),(i_n)_\infty\big),\frac{r_k}{c_1}\Big)
\Big)=1,
\]
the mass transference principle yields
\begin{align*}
\dim_H
\limsup_{k\in\mathbb{N},\,\underline{i}=(i_1,\dots,i_n)\in\Lambda_{r_k/c_1}}
B\Big(\pi\big((i_1,\dots,i_n),(i_n)_\infty\big),
\Big(\frac{r_k}{c_1}\Big)^{(1+2\varepsilon)\delta_{\eta}}\Big)
\geq
\frac{\dim_H K-\varepsilon}{(1+2\varepsilon)\delta_{\eta}}.
\end{align*}

Note that for any $\underline{i}\in\Lambda_{r_k/c_1}$ one has
$|f_{\underline{i}}(K)|\geq r_k$.
Hence, by monotonicity of $\eta$, for all sufficiently large $k\in\mathbb{N}$,
\begin{equation}
\label{EquMono}
\eta(|f_{\underline{i}}(K)|)
\geq \eta(r_k)
\geq r_k^{\delta_{\eta}(1+\varepsilon)}
\geq \Big(\frac{r_k}{c_1}\Big)^{\delta_{\eta}(1+2\varepsilon)} .
\end{equation}

This implies that
\begin{align*}
\dim_H
\limsup_{\underline{i}=(i_1,\dots,i_n)\in\Lambda^*}
\bigcup_{0\leq k\leq n-1}
B\Big(\pi\big((i_1,\dots,i_n),(i_{k+1},\dots,i_n)_\infty\big),
\eta(|f_{\underline{i}}(K)|)\Big)
\geq
\frac{\dim_H K-\varepsilon}{(1+2\varepsilon)\delta_{\eta}}.
\end{align*}
Letting $\varepsilon\to 0$ concludes the proof of Theorem~\ref{ThmUPAWSC}.

\section{Some perspectives}
\label{Secpersp}

\subsection{About the connection between equidistributuion mod q and Mahler's problem}
A careful reader will notice, that in order to prove Theorem \ref{MahlConjO}, one would only need to prove Conjecture \ref{ConjectureO} for mappings $f_{\epsilon}(q)=\frac{\varepsilon}{\sqrt{\log\log \log q}},$ where $\varepsilon>0,$ this rate being given Theorem \ref{ThmBourg}. Thus a more precise conjecture, yielding the conclusion of Theorem \ref{MahlConjO}, is the following.

\begin{conjecture}
Let $b\geq 2$ be an integer. There exists a non increasing  mapping $f:\mathbb{N}\to \mathbb{R}_+$ such that:
\begin{itemize}
\item[(1)] for every $q\in\mathbb{N},$ $q\wedge b=1,$ $$\Big(O_{q}(b)\geq q^{f(q)}\Big)\Rightarrow \Big(\Big\vert\max_{p\in\mathbb{Z}_q^*}\sum_{0\leq k\leq O_q(b)-1}\exp(2i\pi\frac{pb^k}{q})\Big\vert \leq O_{q}(b)o_{q\to +\infty}(1)\Big),$$
\item[(2)] one has $$\lim_{x\to +\infty}\frac{\log \#\left\{1\leq q \leq x, q\wedge b=1 : \ O_q(b)\leq q^{f(q)}\right\}}{\log x}=0.$$

\end{itemize}

\end{conjecture}

\subsection{Extensions of Theorem \ref{ThmMahlerConj} to more general rational IFSs}

%

Regarding IFSs of the form
$T=\left\{f_i(\cdot)=r_i \cdot +r'_i\right\}_{1\leq i\leq m}$,
where $r_i \in (-1,1)\cap \mathbb{Q}\setminus \{0\}$, which can always be rewritten as
\[
T=\left\{f_i(\cdot)=\frac{t_i \cdot}{q_i}+\frac{p_i}{q_i}\right\}_{1\leq i\leq m},
\]
with $q_1,\dots,q_m,t_1,\dots,t_m,p_1,\dots,p_m \in\mathbb{Z}$ and $1\leq |t_i|<|q_i|$ for every $1\leq i\leq m$. For such an IFS, it is easily checked any ultimately periodic sequence produces a rational. But, in these  settings, it follows from \cite{ShmidtSalem} that it is no longer true in general that every  rational in such fractals   admits an eventually periodic coding. We refer to the Appendix, Subsubsection \ref{SubsecRatCod} for more details about this fact. Therefore, studying this problem requires a new description of the possible codings of rationals in this setting.

\subsection{A more refined full-measure statement for homogeneous rational IFSs}

Let
$T=\left\{f_i(\cdot)=\frac{\cdot}{q}+\frac{p_i}{q}\right\}$,
where $q\in \mathbb{N}$ and $p_1,\dots,p_m \in\mathbb{Z}$, and let $K_T$ denote its attractor. As mentioned in Section~\ref{SecSS}, these IFSs always satisfy the weak separation condition, and hence are Ahlfors-regular (see Proposition \ref{WSCAlfh}). By \cite{osc=sosc}, this implies that
\[
\dim_H K_T =\frac{\log m}{\log q}
\quad\Longleftrightarrow\quad
T \text{ satisfies the open set condition}.
\]

In the case where $T$ satisfies the open set condition, it was proved in \cite{TWWTriad,Baker} that, denoting by $\mu_T$ the Ahlfors-regular self-similar measure on $K_T$, for any non-increasing $\psi:\mathbb{N}\to \mathbb{R}_+$ and setting
\[
E_{\psi, T, \text{int}}:=\Big\{x\in K_T : |x-r|\leq \psi(h_{\text{int}}(r))\ \text{for infinitely many } r\in \mathbb{Q}\cap K_T\Big\},
\]
one has
\begin{equation*}
\begin{cases}
\mu_T(E_{\psi, T, \text{int}})=1 \iff \displaystyle\sum_{n\geq 1} n m^n \psi(q^n)^{\dim_H K_T}=+\infty,\\[1ex]
\mu_T(E_{\psi,T,\text{int}})=0 \iff \displaystyle\sum_{n\geq 1} n m^n \psi(q^n)^{\dim_H K_T}<+\infty.
\end{cases}
\end{equation*}

A motivated reader can check that the argument we provide in Section~\ref{Secproofupper}, in this particular case (replacing the box-dimension argument by the use of the Ahlfors-regular measure), only yields that
\[
\sum_{n\geq 1} n^2 m^n \psi(q^n)^{\dim_H K_T}<+\infty
\quad\Longrightarrow\quad
\mathcal{H}^{\dim_H K_T}(E_{\psi,T,\text{int}})=0 .
\]

Thus, we raise the following question: is it true that
\begin{equation*}
\mathcal{H}^{\dim_H K_T}(E_{\psi,T,\text{int}})>0
\iff \displaystyle\sum_{n\geq 1} n m^n \psi(q^n)^{\dim_H K_T}=+\infty?
\end{equation*}

Finally, we mention that another class of self-similar IFSs known to satisfy the WSC is given by Bernoulli convolutions associated with inverses of Pisot numbers. In this particular case, although the attractor is an interval, it is conjectured (see \cite{Bremont}) that all self-similar measures are non-Rajchman, which implies in particular that these measures are singular with respect to Lebesgue measure. Thus, in the above statement, one might not, in general, be able to replace the Hausdorff measure $\mathcal{H}^{\dim_H K_T}$ by a self-similar measure of full dimension.

\section{Appendix}

\subsection{Rational numbers in self-similar sets}
\label{SubSecAppRat}
This sub-section aims at justifying that it is not generic in general for self-similar sets to contain rationals (this is the subject of the next subsubsection) and to make a short state of the art of the existence and possible codings of rational in  self-similar sets associated with self-similar IFS defined by rational parameters.
\subsubsection{Rational numbers in  translations of one-dimensional homogeneous fractals}

Fix $\lambda\in(-1,1)\setminus\{0\}$ and $a_1,\dots,a_m\in\mathbb{R}$, and write
$T=\{f_i(\cdot)=\lambda \cdot+a_i\}_{1\leq i\leq m}$.
Let $K$ be the attractor of $T$. Recall that
\[
K=\Big\{\sum_{k\geq 1}a_{i_k}\lambda^{k-1}\Big\}_{(i_k)_{k\in\mathbb{N}}\in\{1,\dots,m\}^\mathbb{N}}.
\]

Fix $y\in\mathbb{R}$ and write
$T_y=\{f_i(x)=\lambda x+a_i+(1-\lambda)y\}_{1\leq i\leq m}$.
The attractor $K_y$ of $T_y$ is given by
\[
K_y
=
\Big\{\sum_{k\geq 1}(a_{i_k}+(1-\lambda)y)\lambda^{k-1}\Big\}_{(i_k)_{k\in\mathbb{N}}\in\{1,\dots,m\}^\mathbb{N}}
=
y+K.
\]

\begin{proposition}
\[
\begin{cases}
\text{if }\mathcal{L}(K)>0,\text{ then for a.e.\ }y\in\mathbb{R},\ \mathbb{Q}\cap(y+K)\neq\emptyset,\\
\text{if }\mathcal{L}(K)=0,\text{ then for a.e.\ }y\in\mathbb{R},\ \mathbb{Q}\cap(y+K)=\emptyset.
\end{cases}
\]
\end{proposition}

\begin{proof}
Note that $y+K\cap\mathbb{Q}\neq\emptyset$ if and only if
\[
y\in\bigcup_{q\in\mathbb{Q}}(q-K).
\]
If $\mathcal{L}(K)=0$, then
\[
\mathcal{L}\Big(\bigcup_{q\in\mathbb{Q}}(q-K)\Big)=0,
\]
which proves the second statement.

Assume now that $\mathcal{L}(K)>0$.
By  Besicovitch density theorem \cite[Theorem 3]{Be}, there exist $x\in K$ and $r_x>0$ such that
for any $0<r\leq r_x$,
\[
\frac{\mathcal{L}(B(x,r)\cap K)}{\mathcal{L}(B(x,r))}\geq \frac12.
\]
Fix $z\in\mathbb{R}$ and $r\leq r_x/2$.
Since $\{q-x\}_{q\in\mathbb{Q}}$ is dense, there exists $q\in\mathbb{Q}$ such that
$q-x\in B(z,r/2)$.
Hence
\[
B(q-x,r/2)\cap(q-K)\subset B(z,r)\cap(q-K),
\]
and
\begin{align*}
\mathcal{L}\big(B(q-x,r/2)\cap(q-K)\big)
&=\mathcal{L}\big(B(x,r/2)\cap K\big)
\geq \frac{r}{4}
= \frac14\,\mathcal{L}(B(z,r)).
\end{align*}
Therefore, for any $z\in\mathbb{R}$ and any $r\leq r_x$,
\[
\mathcal{L}\Big(B(z,r)\cap\bigcup_{q\in\mathbb{Q}}(q-K)\Big)
\geq \frac14\,\mathcal{L}(B(z,r)),
\]
which implies, again by  Besicovitch density theorem \cite[Theorem 3]{Be}, that 
\[
\mathcal{L}\Big(\mathbb{R}\setminus\bigcup_{q\in\mathbb{Q}}(q-K)\Big)=0.
\]
This concludes the proof.
\end{proof}

\subsubsection{Rationals and self-similar IFS's defined by  rational parameters}
\label{SubsecRatCod}
Let $\lambda_1,...,\lambda_m \in \mathbb{Q}\cap \Big((-1,1)\setminus \left\{0\right\}\Big),$ $a_1,....,a_m \in \mathbb{Q}$ be rationals and set $$T=\left\{f_i(\cdot)=\lambda_i \cdot +a_i\right\}_{1\leq i\leq m}.$$ For every $n\in\mathbb{N},$ for every $0\leq k\leq n-1,$ for every $(i_1,...,i_n)\in\left\{1,...,m\right\}^n,$ 

\begin{align*}
&\pi\Big((i_1,...,i_n),(i_{k+1},...,i_n)_{\infty}\Big)\\
&=\sum_{1\leq t\leq n} \Big(\prod_{1\leq j<t}\lambda_{i_j}\Big)a_{i_t}+ \Big(\prod_{1\leq j \leq n}\lambda_{i_j}\Big)\frac{\sum_{k+1 \leq t \leq n}\Big(\prod_{k+1 \leq j <t}\lambda_{i_j}\Big)a_{i_j}}{1 -\prod_{k+1 \leq j \leq n}\lambda_{i_j}}\in \mathbb{Q}.
\end{align*}

Hence  $K_T$ always contains rational numbers. It is not true in general that, for any such $T$, rationals in $ K_T$ corresponds to projections of ultimately periodic sequences. This is due to the following result.
\begin{theoreme}[]
Let $\lambda\in (0,1)$ and $T=\left\{0,...,\lfloor \frac{1}{\lambda}\rfloor \right\}.$ Then 
\begin{align*}
&\mathbb{Q}(\lambda)\cap K_T \subset\left\{\pi\Big((i_1,...,i_n),(i_{k+1},...,i_n)_{\infty}\Big)\right\}_{0\leq k\leq n-1, n \in\mathbb{N}, (i_1,...,i_n)\in\left\{1,...,m\right\}^n}\\
&\Longrightarrow (\frac{1}{\lambda}\text{ is Pisot or Salem}).
\end{align*}

\end{theoreme}
In particular, for $\lambda=\frac{2}{3},$ $\mathbb{Q}(\lambda)=\mathbb{Q}$ and $\frac{1}{\lambda}$ is not an algebraic integer  (thus is not Pisot nor Salem). Hence there must exist $r\in\mathbb{Q}\cap K_T$ which does not admit an ultimately periodic coding.

As mentioned in the introduction, in the special case where $T$ is of the form $T=\left\{f_i(\cdot)=\frac{\cdot}{q_i}+\frac{p_i}{q_i}\right\}_{1\leq i\leq m},$ where $q_i,p_i$ are integers, it was proved by Fishman and Simmons , when $T$ satisfies the open set condition, that the rationals 
in $K_T$ corresponds to eventually periodic sequences. Later on, in  \cite{Baker}, Baker remarked (but did not prove explicitly) that the assumption on the separation condition can be removed. For the sake of completeness, we prove this fact. 
\begin{proposition}[\cite{FishSim,Baker}]
\label{ThmRatCod}
Let $q_1,...,q_m \geq 2$ and $p_1,...,p_m \in\mathbb{Z}$ be integers and set $T=\left\{f_i(\cdot)=\frac{\cdot}{q_i}+\frac{p_i}{q_i}\right\}.$ Then $r\in\mathbb{Q}\cap K_T$ if and only if there exists $n\in\mathbb{N},$ $0\leq k\leq n-1$ and $(i_1,...,i_n)\in\left\{1,...,m\right\}^n$ such that $$x=\pi\Big((i_1,...,i_n,(i_{k+1},...,i_n)_{\infty})\Big).$$
\end{proposition}
\begin{proof}
As $T$ has rational parameters, one only needs to prove that if $r\in\mathbb{Q}\cap K_T,$ there exists $n\in\mathbb{N},$ $0\leq k\leq n-1$ and $(i_1,...,i_n)\in\left\{1,...,m\right\}^n$ such that $$r=\pi\Big((i_1,...,i_n,(i_{k+1},...,i_n)_{\infty})\Big).$$
Fix $r=\frac{a}{b}\in\mathbb{Q}\cap K_T$ and  $(j_n)_{n\in\mathbb{N}}\in\left\{1,...,m\right\}^\mathbb{N}$ such that $$\pi\Big((j_n)_{n\in\mathbb{N}}\Big)=r.$$
For any $1\leq i\leq m,$ 
$$f_i^{-1}(r)=q_i r -p_i \in\frac{\mathbb{Z}}{b}.$$
Hence, for any $t\in\mathbb{N},$ $f_{j_t}^{-1}\circ...\circ f_{j_1}^{-1}(r)\in K_T \cap \frac{\mathbb{Z}}{b}.$
Thus there exists $n\in\mathbb{N}$ and $0\leq k\leq n-1$ such that $$f_{j_n}^{-1}\circ...f_{j_i}^{-1}(r)=f_{j_k}^{-1}\circ...f_{j_i}^{-1}(r).$$
In particular, $x=f_{j_k}^{-1}\circ...\circ f_{j_1}^{-1}(r)$ is a fixed point of $f_{j_{n}}^{-1}\circ...\circ f_{j_{k+1}}^{-1},$ hence is a fixed point of $f_{j_{k+1}}\circ...\circ f_{j_n}.$ Writing $g=f_{j_{k+1}}\circ...\circ f_{j_n},$ $$\lim_{t\to +\infty}g^{(t)}(0)=x,$$
where $g^{(t)}$  is $g$ composed $t$ times. This yields $$x=\pi\Big((j_{k+1},...,j_n)_{\infty}\Big)\Rightarrow r=\pi\Big((j_1,...,j_k,(j_{k+1},...,j_n)_{\infty})\Big),$$
which proves the claim.
\end{proof}

\subsection{WSC and Ahlfors regularity}

Let $T=\{f_1,\dots,f_m\}$ be a self-similar IFS satisfying WSC
(Definition~\ref{DefWSC}), let $K_T$ be its attractor, and let
$0<c_1\leq\cdots\leq c_m<1$ be the contraction ratios of $f_1,\dots,f_m$.

We prove Proposition~\ref{WSCAlfh}.

Fix $x\in K_T$ and $0<r\leq c_1$. Then there exists
$\underline{i}\in\bigcup_{k\geq 1}\{1,\dots,m\}^k$ such that
\[
\begin{cases}
f_{\underline{i}}(K)\subset B(x,r),\\
|f_{\underline{i}}(K)|\geq c_1 r.
\end{cases}
\]
Since $f_{\underline{i}}$ is a homothetic map with contraction ratio
$c_{\underline{i}}=|f_{\underline{i}}(K)|/|K|$, recalling
Theorem~\ref{ThmZerner}, we obtain
\[
\mathcal{H}^{\dim_H K_T}(B(x,r))
\geq c_{\underline{i}}^{\dim_H K_T}\mathcal{H}^{\dim_H K_T}(K_T)
\geq \kappa_1 r^{\dim_H K_T}.
\]

On the other hand, since
$$K_T\cap B(x,r)\subset\bigcup_{\underline{i}\in\Lambda_{B(x,r)}}f_{\underline{i}}(K),$$
we have
\begin{align*}
\mathcal{H}^{\dim_H K_T}(B(x,r))
&\leq
\kappa_2
\sup_{y\in K_T,\rho>0}\#\Lambda_{B(y,\rho)}\,
\mathcal{H}^{\dim_H K_T}(K_T)\, r^{\dim_H K_T}\\
&\leq \kappa_3 r^{\dim_H K_T}.
\end{align*}

\subsection{About Remark \ref{Remequi}}
In this section, we prove the following.
\begin{proposition}
Let $T$ be a self-similar IFS of the form $$T=\Big\{f_i(\cdot)=\frac{\cdot}{b}+\frac{p_i}{b}\Big\}_{1\leq i\leq m},$$ where $b\geq 2$ and $p_1,....,p_m \in\mathbb{Z} $ and write $K_T$ its attractor. The three following items are  equivalent:
\begin{itemize}
\item[(1)] $\mathcal{L}(K_T)>0,$
\medskip
\item[(2)] $K_T$ has non empty interior,
\medskip
\item[(3)] $\dim_H K_T =1.$
\end{itemize}
\end{proposition}
\begin{proof}
Call again $\widetilde{K}_T \subset \mathbb{T}^1$ the projection of $K_T$ on $\mathbb{T}^1$ and recall that $\widetilde{K}_T$ is $\times b$ invariant. Clearly $(2)\Rightarrow (1).$ 

We prove $(1)\Rightarrow (2):$ Assume that $\mathcal{L}(K_T)>0$ and notice that  $\mathcal{L}(K_T)>0 \Rightarrow \mathcal{L}(\widetilde{K}_T)>0.$ Since $\mathcal{L}$ is $\times b$ ergodic, there exists  $x\in \mathbb{T}^1$ such that $$\lim_{n\to +\infty}\frac{\sum_{k=0}^{n-1}\chi_{\widetilde{K}_T}(b^k x)}{n}=\mathcal{L}(\widetilde{K}_T),$$
where $\chi_{\widetilde{K}_T}$ is the indicator function of $\widetilde{K}_T$. As $y\in \widetilde{K}_T \Rightarrow by \in \widetilde{K}_T,$ one has $$\lim_{n\to +\infty}\frac{\sum_{k=0}^{n-1}\chi_{\widetilde{K}_T}(b^k x)}{n}=1$$
so that $\mathcal{L}(\widetilde{K}_T)=1$. $\widetilde{K}_T$ being compact, this yields $\widetilde{K}_T =\mathbb{T}^1$ and  $K_T$ has non empty interior. 

As $(1)\Rightarrow (3)$  and by Proposition \ref{WSCAlfh}, $K_T$ is $\dim_H K_T$-Alfhors regular, $(1)\Leftrightarrow (3).$  
\end{proof}

\textbf{Acknowledgments:} The author would like to thank J. Barral and  S. Seuret  for their valuable advice regarding this manuscript, as well as B. Wang for pointing out a mistake in the $``$ proof$"$ of Conjecture \ref{ConjectureO}.
\bibliographystyle{plain}
\bibliography{bibliogenubi}

\begin{thebibliography}{10}

\bibitem{Baker}
S.~Baker.
\newblock Intrinsic {D}iophantine approximation for overlapping iterated
  function systems.
\newblock {\em Math.Ann.}, 388:3259–3297, 2024.

\bibitem{BFmult}
J.~Barral and D.~Feng.
\newblock On multifractal formalism for self-similar measures with overlaps.
\newblock {\em Math. Z.}, 298:359--383, 2021.

\bibitem{KintFract}
T.~Benard, W.~He, and H.~Zhang.
\newblock Khintchine dichotomy for self-similar measures.
\newblock {\em to appear in Journal of the American Mathematical Society},
  2025.

\bibitem{BV}
V.~Beresnevitch and S.~Velani.
\newblock A mass transference principle and the {D}uffin-{S}chaeffer conjecture
  for {H}ausdorff measures.
\newblock {\em Ann. Math.}, 164(3), 2006.

\bibitem{Be}
A.S. Besicovitch.
\newblock A general form of the covering principle and relative differentiation
  of additive functions.
\newblock {\em Math. Proc. Camb. Phil. Soc.}, 41:103–110, 1945.

\bibitem{BourgainSumExp}
J.~Bourgain.
\newblock Exponential sum estimates over subgroups of $\mathbb{Z}_q^*$, q
  arbitrary.
\newblock {\em Journal d'analyse mathématique}, 97:317–355, 2005.

\bibitem{Bremont}
J.~Bremont.
\newblock Self-similar measures and the {R}ajchman property.
\newblock {\em Annales Henri Lebesgue}, 4:973--1004, 2021.

\bibitem{Bugdu}
Y.~Bugeaud.
\newblock Metric {D}iophantine approximation on the middle-third {C}antor set.
\newblock {\em J. Europ. Math. Soc.}, 2016.

\bibitem{YuVarjuGene}
S.~Chen.
\newblock The hausdorff dimension of the intersection of $\psi$-well
  approximable numbers and self-similar sets.
\newblock {\em arXiv:2510.17096}, 2025.

\bibitem{YuVarju}
S.~Chow, P.~Varjù, and H.~Yu.
\newblock Counting rationals and diophantine approximation in missing-digit
  {C}antor sets.
\newblock {\em arXiv:2402.18395}, 2024.

\bibitem{ED4}
E.~Daviaud.
\newblock Dynamical {D}iophantine approximation and shrinking targets for
  ${C^1}$ weakly conformal {IFS}s with overlaps.
\newblock {\em Ergodic theory and dynamical systems}, 45, issue 6:1777--1826,
  2023.

\bibitem{ED3}
E.~Daviaud.
\newblock A dimensional mass transference principle for {B}orel probability
  measures and applications.
\newblock {\em Advances in mathematics}, 474:47 pp., 2025.

\bibitem{ErdTur}
P~Erdös and P.~Turàn.
\newblock On a problem in the theory of uniform distribution. {II}.
\newblock {\em Indagationes Mathematicae}, 51:1262–1269, 1948.

\bibitem{Dimh=DimB}
K.~Falconer.
\newblock Dimension of quasi-self-similar sets.
\newblock {\em Proc.Amer.Math.Soc}, 106:543--554, 1989.

\bibitem{F}
K.~Falconer.
\newblock {\em Fractal geometry}.
\newblock John Wiley \& Sons, Inc., Hoboken, NJ, second edition, 2003.
\newblock Mathematical foundations and applications.

\bibitem{Feng2007}
D.~Feng.
\newblock Gibbs properties of self-conformal measures and the multifractal
  formalism.
\newblock {\em Ergodic Theory Dynam. Systems}, 27(3):787--812, 2007.

\bibitem{FH}
D.~Feng and H.~Hu.
\newblock Dimension theory of iterated function systems.
\newblock {\em Comm. Pure Appl. Math.}, 62:1435--1500, 2009.

\bibitem{DefWSC}
D.J. Feng and K.S. Lau.
\newblock Multifractal formalism for self-similar measures with weak separation
  condition.
\newblock {\em Journal de mathématiques pures et appliquées}, 92(4):407--428,
  2009.

\bibitem{FishSim}
S.~Fishman and D.~Simmons.
\newblock Intrinsic approximation for fractals defined by rational iterated
  function systems: Mahler's research suggestion.
\newblock {\em Proc.Lond.Math.Soc.}, 109, 2014.

\bibitem{HeLiaoJarnik}
Y.~He and L.~Liao.
\newblock Jarník-type theorem for self-similar sets.
\newblock {\em arXiv:2602.01307}, 2026.

\bibitem{Hutchinson}
J.E. Hutchinson.
\newblock Fractals and self similarity.
\newblock {\em Indiana Univ. Math. J.}, 30:713--747, 1981.

\bibitem{Maynkou}
D.~Koukoulópoulos and J.~Maynard.
\newblock On the {D}uffin-{S}haeffer conjecture.
\newblock {\em Ann. Math.}, 192:251--307, 2020.

\bibitem{LSV}
J.~Levesley, C.~Salp, and S.~Velani.
\newblock On a problem of {K}. {M}ahler: Diophantine approximation and {C}antor
  sets.
\newblock {\em Math.Ann.}, 338:97--118, 2007.

\bibitem{Mahler}
K.~Mahler.
\newblock Some suggestions for further research.
\newblock {\em Bull.Austral.Math.Soc}, 29:101--108, 1984.

\bibitem{osc=sosc}
A.~Schief.
\newblock Separation properties for self-similar sets.
\newblock {\em Proc. Amer. Math. Soc.}, 122, 1, 1994.

\bibitem{ShmidtSalem}
K.~Schmidt.
\newblock On peridic expansions of {P}isot numbers and {S}alem numbers.
\newblock {\em Bull.Lond.Math.Soc}, 12(4):269--278, 1980.

\bibitem{TWWTriad}
B.~Tan, B.~Wang, and J.~Wu.
\newblock Mahler's question for intrinsic diophantine approximation on triadic
  {C}antor set: the divergence theory.
\newblock {\em Math. Z.}, 2021.

\bibitem{ZernerWSC}
M.P.W. Zerner.
\newblock Weak separation properties for self-similar sets.
\newblock {\em Proc.Amer.Math.Soc.}, 124:3529--3539, 1996.

\end{thebibliography}

\end{document}